\documentclass[11pt,reqno]{amsart}

\usepackage[T1]{fontenc}
\usepackage{lmodern}
\usepackage{amsmath,amssymb,amsthm,mathtools,bm}
\usepackage{graphicx}
\usepackage{microtype}
\usepackage{enumitem}
\usepackage{url}
\usepackage[hidelinks]{hyperref}

\allowdisplaybreaks
\theoremstyle{plain}
\newtheorem{lemma}{Lemma}
\newtheorem{theorem}{Theorem}
\newtheorem*{corollary}{Corollary}

\theoremstyle{definition}
\newtheorem{definition}{Definition}
\newtheorem{remark}{Remark}
\newtheorem{example}{Example}

\newcommand{\R}{\mathbb{R}}

\newcommand{\Z}{\mathbb{Z}}
\newcommand{\cA}{\mathcal{A}}
\newcommand{\cC}{\mathcal{C}}
\newcommand{\cF}{\mathcal{F}}
\newcommand{\cH}{\mathcal{H}}
\newcommand{\cL}{\mathcal{L}}
\newcommand{\cM}{\mathcal{M}}
\newcommand{\cN}{\mathcal{N}}
\newcommand{\Om}{\Omega}
\newcommand{\Upsilonq}{\Upsilon}
\newcommand{\rank}{\operatorname{rank}}
\newcommand{\codim}{\operatorname{codim}}
\newcommand{\spann}{\operatorname{span}}
\newcommand{\const}{\operatorname{const}}

\title[Analytical solutions of a continuous analogue]{ANALYTICAL SOLUTIONS OF ONE CLASS OF VOLTERRA QUADRATIC STOCHASTIC OPERATORS OF A TWO-SEX POPULATION WITH CONTINUOUS TIME}

\author{Rasulov Xaydar Raupovich}

\address{Bukhara State University, Bukhara, Uzbekistan}

\address{V.I.Romanovskiy Institute of Mathematics, Uzbekistan Academy of Science, Tashkent, Uzbekistan}

\email{xrasulov71@mail.ru}


\begin{document}

\begin{abstract}
This paper investigates analytical solutions of a continuous analogue of a class of Volterra quadratic stochastic operators of a two-sex population using the homotopy analysis method. The positive invariance of the state space of the system is established. By means of Lyapunov functions, inclusions for the $\omega$-limit sets are obtained, together with exponential decay estimates for individual components of the trajectories under the corresponding conditions on the coefficients of the system. Sufficient conditions are obtained for the coincidence of the $\omega$-limit sets of the discrete and continuous models. For the diagonal invariant subset, sufficient conditions are established for the periodicity of all interior non-equilibrium trajectories. By means of a majorant series, a sufficient condition for the absolute and uniform convergence of the homotopy series on $[0,+\infty)$ is obtained under the Hurwitz stability of the reduced Jacobian of the system. In addition, explicit analytical solutions of the Cauchy problem are found for special relations between the heredity coefficients. Computational experiments show a rapid decay of the terms of the homotopy series and good approximation accuracy already when using the first few terms of the expansion.
\end{abstract}

\maketitle

\section{Introduction}

Quadratic stochastic operators (QSOs) describe the evolution of probability distributions in which the new state of the system depends quadratically on the current state. Such models originally arose in population genetics to describe the dynamics of genotype distributions under Mendelian inheritance, and subsequently, owing to their natural biological interpretation and close connection with the theory of nonlinear dynamical systems, they found application in other areas as well.

The foundations of the discrete theory of QSOs were laid by S.~N.~Bernstein \cite{ref1}. Its further development led to the study of models of two-sex populations, systematically presented in the monograph of Yu.~I.~Lyubich \cite{ref2}. Contemporary studies of QSOs and two-sex QSOs are presented in \cite{ref3}-\cite{ref10}, while bifurcations of a discrete epidemic model are considered in \cite{ref11}.

In \cite{ref10}, special subclasses of the two-sex QSOs introduced in \cite{ref2} were singled out, for which substantial results on the dynamics of the corresponding operators were obtained. Let us give the definition needed for the further exposition.

Let
\[
S^{n-1}=\left\{x=(x_1,\ldots,x_n)\in\R^n:\ x_i\ge0,\ \sum_{i=1}^n x_i=1\right\}
\]
be the $-(n-1)$-dimensional simplex.

A state of the population is a pair of probability distributions
\[
x=(x_1,\ldots,x_n)\in S^{n-1},\qquad
y=(y_1,\ldots,y_\nu)\in S^{\nu-1},
\]
defined, respectively, on the sets of female types $\cF=\{\cF_1,\ldots,\cF_n\}$ and male types $\cM=\{\cM_1,\ldots,\cM_\nu\}$. The state space of the population is denoted by
\[
\Om=S^{n-1}\times S^{\nu-1}.
\]

For parents of types $\cF_i$ and $\cM_k$ the coefficient $p_{ik,j}^{(f)}$ equals the probability of birth of a female offspring of type $\cF_j$, while $p_{ik,l}^{(m)}$ is the probability of birth of a male offspring of type $\cM_l$. These coefficients satisfy the conditions
\begin{equation}
\label{eq:1}
\begin{cases}
 p_{ik,j}^{(f)}\ge0,\qquad \displaystyle\sum_{j=1}^n p_{ik,j}^{(f)}=1,\\[1ex]
 p_{ik,l}^{(m)}\ge0,\qquad \displaystyle\sum_{l=1}^{\nu} p_{ik,l}^{(m)}=1.
\end{cases}
\end{equation}
The coefficients may account for recombination, gamete selection, mutation, and differences in fertility, provided the corresponding mechanisms are compatible with the assumed form of the operator.

If $(x,y)$ is the state of the population in the current generation and $(x',y')$ is the state of the next generation, then the evolution operator $W$ has the form
\begin{equation}
\label{eq:2}
W:=
\begin{cases}
\displaystyle x_j'=\sum_{i=1}^n\sum_{k=1}^{\nu}p_{ik,j}^{(f)}x_i y_k, & 1\le j\le n,\\[2ex]
\displaystyle y_l'=\sum_{i=1}^n\sum_{k=1}^{\nu}p_{ik,l}^{(m)}x_i y_k, & 1\le l\le \nu.
\end{cases}
\end{equation}

\begin{definition}
Operator \eqref{eq:2} is called a Volterra quadratic stochastic operator of a two-sex population (VQSOBP) if
\begin{equation}
\label{eq:3}
\begin{cases}
 p_{ik,j}^{(f)}=0, & \text{if }j\notin\{i,k\},\\
 p_{ik,l}^{(m)}=0, & \text{if }l\notin\{i,k\}.
\end{cases}
\end{equation}
\end{definition}

Conditions \eqref{eq:3} mean that the type index of an offspring may coincide only with the type index of one of the parents. From \eqref{eq:1} and \eqref{eq:3} it follows that
\[
p_{ii,i}^{(f)}=p_{ii,i}^{(m)}=1,\qquad 1\le i\le n.
\]
Moreover, $p_{ik,i}^{(f)}=1$, $1\le i\le n<k\le\nu$, since an index $k>n$ does not correspond to a female type.

An arbitrary VQSOBP has the form \cite{ref10}:
\begin{equation}
\label{eq:4}
W:=
\begin{cases}
\displaystyle
x_j'=x_j\left(1+\sum_{\substack{k=1\\k\ne j}}^{\nu}(p_{jk,j}^{(f)}-1)y_k\right)
+y_j\left(\sum_{\substack{i=1\\i\ne j}}^n p_{ij,j}^{(f)}x_i\right), & 1\le j\le n,\\[3ex]
\displaystyle
y_l'=y_l\left(1+\sum_{\substack{i=1\\i\ne l}}^n(p_{il,l}^{(m)}-1)x_i\right)
+x_l\left(\sum_{\substack{k=1\\k\ne l}}^{\nu}p_{lk,l}^{(m)}y_k\right), & 1\le l\le n,\\[3ex]
\displaystyle
y_l'=y_l\left(1+\sum_{i=1}^n(p_{il,l}^{(m)}-1)x_i\right), & n<l\le\nu.
\end{cases}
\end{equation}
Operator \eqref{eq:4} maps $\Om$ into itself, and its iterates
\[
z^{k+1}=W(z^k),\qquad k\in\Z_+,\qquad z=(x,y),
\]
define a discrete dynamical system describing the change of the population's state from generation to generation.

In this paper we study the analytical solutions of the continuous analogue of operator \eqref{eq:4}, defined on $\Om$ by the equality
\[
\dot z(t)=W(z(t))-z(t),\qquad t\ge0.
\]
The continuous analogue is an independent dynamical system and, generally speaking, is not equivalent to the original discrete model. The coincidence of the fixed points of the operator with the equilibria of the continuous system does not imply dynamical equivalence: local stability, periodic regimes, and the global structure of trajectories require separate investigation and comparison (see Lemmas~7--9 and Section~6).

Systems of nonlinear ordinary differential equations (ODEs) arise naturally not only as continuous analogues of discrete population operators, but also directly in models of population dynamics, epidemiology, and other applied problems. For such systems, explicit analytical solutions can usually not be obtained, and therefore a substantial part of their investigation is based on methods of the qualitative theory of dynamical systems.

Thus, in \cite{ref12}, for a nonlinear ``predator--prey'' type system, the stationary points, their stability, phase structure, and bifurcations were investigated. Stochastic population models with regime switching, including the Lotka--Volterra system and a single-species model with the Allee effect, were studied in \cite{ref13,ref14}, where sufficient conditions for extinction and persistence of populations were obtained, along with conditions for positive recurrence, ergodicity, and the existence of stationary distributions.

In \cite{ref15,ref16}, a qualitative analysis of epidemiological models of the spread of the Ebola virus was carried out: equilibrium states and their stability were investigated, threshold conditions for the dynamics of the disease depending on the basic reproduction number $R_0$ were obtained, and bifurcations together with the sensitivity of $R_0$ to model parameters were studied. For planar polynomial systems, conditions for the absence of periodic orbits and estimates for the number of limit cycles in the presence of invariant algebraic curves were obtained in \cite{ref17}.

Qualitative analysis describes the stability and limiting behavior of a system, but usually does not give the explicit form of its trajectories. It is therefore natural to seek exact solutions and to construct convergent analytical representations of the trajectories of nonlinear systems of ODEs.

Various analytical approaches have been developed. In \cite{ref18}, based on spectral expansions of the Koopman operator \cite{ref19}, a connection was established between the topological structure of the phase space and the existence of globally stable attractors of the dynamical system. In \cite{ref20}, an analytical solution of a system of nonlinear differential equations was constructed in the reproducing-kernel space $W_2^2[0,1]$: the solution is represented as a series, and the convergence of its partial sums to the exact solution is proved.

In \cite{ref21}, an approach to constructing exact solutions of nonlinear systems of ODEs was proposed, based on expanding the nonlinear terms in Adomian polynomials and subsequently summing the resulting series using the multiplicative inverse of a power series. In \cite{ref22}, an analytical method was proposed for solving linear homogeneous ODEs of arbitrary order with variable coefficients, based on reducing the equation to a first-order system of ODEs by using jump-transfer matrices and their limiting differential generalization. In \cite{ref23}, solutions of certain nonlinear ODEs with linear and quadratic right-hand sides are represented in the form of exponential series.

In \cite{ref24}-\cite{ref27}, Volterra and non-Volterra quadratic dynamical systems with continuous time were studied: their equilibrium states were found, phase portraits were constructed, numerical solutions were obtained and compared with analytical solutions, represented in particular by power series with recursively computed coefficients.

For SIR epidemic models, various approaches to constructing analytical solutions have been developed. In particular, exact parametric solutions were obtained, and extended models were reduced to Abel's equation \cite{ref28}. Relations for the final size of an epidemic and their solvability were investigated \cite{ref29}, and analytical approximations were constructed by continuing power series using asymptotic approximants \cite{ref30}. For a special modification of the SIR model, elementary solutions in terms of hyperbolic functions were found \cite{ref31}. These results show that the special structure of a nonlinear system, in a number of cases, makes it possible to obtain an explicit solution or to construct its convergent analytical representation.

One of the widely used analytical methods for solving nonlinear differential equations is the homotopy analysis method (HAM) \cite{ref32}. Its essential feature is the introduction of an auxiliary parameter $\hbar$, which makes it possible to control the region and rate of convergence of the homotopy series. In \cite{ref33}, the method was developed for application to nonlinear ODEs and partial differential equations, using various auxiliary linear operators and error estimates for the approximations. In \cite{ref34}, HAM was applied to the construction of analytical solutions of nonlinear differential equations arising in physical problems.

In \cite{ref35}, using the nonlinear Blasius equation as an example, an optimal version of HAM was proposed, employing several convergence-control parameters. The application of HAM to linear and nonlinear Cauchy problems for ODEs and their systems was considered in \cite{ref36,ref37}, and for a nonlinear system arising in a magnetohydrodynamic flow problem, rapidly convergent homotopy series were constructed \cite{ref38}.

The choice of the parameter $\hbar$ significantly affects the convergence of the homotopy series \cite{ref39,ref40}. As noted in \cite{ref40}, determining a suitable value of the convergence-control parameter $\hbar$ and finding its optimal value ensuring the fastest convergence, in general remain open problems of HAM.

\section{Problem statement}

In this paper we study the analytical solution of the Cauchy problem for the continuous analogue of operator \eqref{eq:4}:
\begin{equation}
\label{eq:5}
\begin{cases}
\displaystyle
\dot x_j(t)=x_j(t)\left(\sum_{\substack{k=1\\k\ne j}}^{\nu}(p_{jk,j}^{(f)}-1)y_k(t)\right)
+y_j(t)\left(\sum_{\substack{i=1\\i\ne j}}^n p_{ij,j}^{(f)}x_i(t)\right), & 1\le j\le n,\\[3ex]
\displaystyle
\dot y_l(t)=y_l(t)\left(\sum_{\substack{i=1\\i\ne l}}^n(p_{il,l}^{(m)}-1)x_i(t)\right)
+x_l(t)\left(\sum_{\substack{k=1\\k\ne l}}^{\nu}p_{lk,l}^{(m)}y_k(t)\right), & 1\le l\le n,\\[3ex]
\displaystyle
\dot y_l(t)=y_l(t)\left(\sum_{i=1}^{n}(p_{il,l}^{(m)}-1)x_i(t)\right), & n<l\le\nu,
\end{cases}
\end{equation}
with initial conditions
\[
x_j(0)=x_j^0,\quad 1\le j\le n,\qquad
y_l(0)=y_l^0,\quad 1\le l\le\nu,\qquad
z^0=(x^0,y^0)\in\Om.
\]

The transition from the discrete system to its continuous analogue links the recursive description with a system of differential equations, making it possible to apply methods of the qualitative theory of ODEs, to find analytical solutions, and to obtain quantitative estimates of the rate of convergence. Comparing the discrete and continuous models also makes it possible to establish which dynamical properties are preserved under such a transition.

If $x_j=0$, then $\dot x_j\ge0$, and if $y_l=0$, then $\dot y_l\ge0$. Therefore system \eqref{eq:5} preserves the non-negativity of the coordinates.

Direct summation of the equations, taking into account \eqref{eq:1} and \eqref{eq:3}, gives
\[
\sum_{j=1}^n\dot x_j(t)=0,\qquad \sum_{l=1}^{\nu}\dot y_l(t)=0.
\]
Consequently, for $z^0\in\Om$ both normalizations are preserved, and $\Om$ is positively invariant. The polynomial nature of the right-hand side, together with the compactness of $\Om$ guarantees the global existence and uniqueness of the solution of the Cauchy problem for $t\ge0$. Thus, system \eqref{eq:5} generates a global semiflow on $\Om$.

\section{Limiting behavior, stability, and periodic regimes}

\subsection{Inclusions for the $\omega$-limit set and stability of equilibria}

In this section we study the limiting behavior of trajectories of system \eqref{eq:5} and the stability of some of its equilibria. The main tool is Lyapunov functions.

By $\omega(z^0)$ we denote the $\omega$-limit set of the continuous trajectory with initial state $z^0$.

\begin{definition}
A continuous function $\Phi$, defined on the positively invariant subset of $\Om$, under consideration, is called a Lyapunov function for system \eqref{eq:5} if, along every trajectory $(x(t),y(t))$ in this subset, the finite limit
\[
\lim_{t\to+\infty}\Phi(x(t),y(t))
\]
exists.
\end{definition}

For $1\le k\le n<j\le\nu$ put
\[
q_{kj}=1-p_{kj,j}^{(m)}.
\]
This quantity characterizes the loss of male type $j$ when a maternal type $k$ is crossed with a paternal type $j$.

\begin{lemma}
Let $n<\nu$. Then
\[
\Phi(x,y)=\prod_{j=n+1}^{\nu}y_j
\]
is a Lyapunov function for system \eqref{eq:5}.

If $b_j\ge0$ and $\sum_{j=n+1}^{\nu}b_j=1$, then
\[
\Phi_b(x,y)=\prod_{j=n+1}^{\nu}y_j^{b_j}
\]
is also a Lyapunov function; factors with $b_j=0$ are not taken into account. If
\[
\mu_b=\min_{1\le k\le n}\sum_{j=n+1}^{\nu}b_jq_{kj}>0,
\]
then
\begin{equation}
\label{eq:6}
\Phi_b(x,y)\le \Phi_b(x^0,y^0)e^{-\mu_b t}.
\end{equation}
If
\[
\mu_{0j}=\min_{1\le k\le n}q_{kj}>0,\qquad j=n+1,\ldots,\nu,
\]
then
\[
y_j\le y_j^0e^{-\mu_{0j}t}
\]
and
\[
\omega(x^0,y^0)\subset S^{n-1}\times\Upsilonq,
\qquad
\Upsilonq=\{y\in S^{\nu-1}:y_{n+1}=\cdots=y_{\nu}=0\}.
\]
\end{lemma}

\begin{proof}
For $j>n$,
\[
\dot y_j=-y_j\sum_{k=1}^n q_{kj}x_k\le0.
\]
Hence $y_j$, and with it $\Phi(x,y)$, is non-increasing and bounded below by zero.

If $y_j>0$ for every $j$ with $b_j>0$, then
\[
\frac{\dot\Phi_b}{\Phi_b}
=-\sum_{k=1}^n x_k\sum_{j=n+1}^{\nu}b_jq_{kj}\le-\mu_b.
\]
If one of the factors involved is initially zero, it remains zero and $\Phi_b\equiv0$. Inequality \eqref{eq:6} follows from Gr\"onwall's lemma. The coordinatewise estimates are obtained analogously and give the stated inclusion for the $\omega$-limit set.
\end{proof}

\begin{lemma}
Let $n=\nu$, $1\le r<n$ and
\[
p_{jk,j}^{(f)}+p_{jk,j}^{(m)}<1,\qquad
p_{kj,j}^{(f)}+p_{kj,j}^{(m)}<1
\]
for $1\le j\le r<k\le n$. Then
\[
\Phi(x,y)=\sum_{j=1}^{r}(x_j+y_j)
\]
is a Lyapunov function. If
\[
\Phi_0=\Phi(x^0,y^0)<2
\]
and
\[
\vartheta=
\min_{\substack{1\le j\le r\\r<k\le n}}
\left\{1-p_{jk,j}^{(f)}-p_{jk,j}^{(m)},\,
1-p_{kj,j}^{(f)}-p_{kj,j}^{(m)}\right\}>0,
\]
then
\[
\Phi(x,y)\le \Phi_0e^{-\vartheta(1-\Phi_0/2)t}.
\]
\end{lemma}

\begin{proof}
Put $X=\sum_{j=1}^r x_j$, $Y=\sum_{j=1}^r y_j$, then
\[
\dot\Phi\le-\vartheta[X(1-Y)+Y(1-X)]
\le-\vartheta\Phi\left(1-\frac{\Phi_0}{2}\right).
\]
Since $\Phi_0<2$ this gives the required exponential estimate.
\end{proof}

\begin{lemma}
Suppose there exist $l\in\{n+1,\ldots,\nu\}$, $q\in\{1,\ldots,n\}$ such that
\[
\max_{1\le i\le n}p_{il,l}^{(m)}
\le
\min_{1\le k\le\nu}p_{qk,q}^{(f)}.
\]
Then, on the relative interior of $\Om$,
\[
\Phi(x,y)=\frac{y_l}{x_q}
\]
is a Lyapunov function.

If
\[
\vartheta_1=
\min_{1\le k\le\nu}p_{qk,q}^{(f)}-
\max_{1\le i\le n}p_{il,l}^{(m)}>0,
\]
then
\[
\dot\Phi<0,\qquad \Phi(x,y)\le\Phi_0e^{-\vartheta_1t}.
\]
\end{lemma}

\begin{proof}
The logarithmic derivative satisfies
\[
\frac{d}{dt}\ln\Phi
\le
\max_{1\le i\le n}p_{il,l}^{(m)}-
\min_{1\le k\le\nu}p_{qk,q}^{(f)}.
\]
The non-strict condition guarantees that $\Phi$, does not increase, while the strict condition gives its exponential decay.
\end{proof}

We now turn to the stability of equilibria and their basins of attraction. In what follows we assume $2\le n\le\nu$.

Introduce
\[
E_q=(e_q^{(n)},e_q^{(\nu)}),\qquad q=1,\ldots,n,
\]
where $e_q^{(n)}$ and $e_q^{(\nu)}$ are the corresponding unit (coordinate) vectors. Set
\[
\begin{aligned}
a_q&=\max_{\substack{1\le k\le n\\k\ne q}}\left(1-p_{qk,q}^{(f)}\right),
&b_q&=\min_{\substack{1\le i\le n\\i\ne q}}p_{iq,q}^{(f)},\\
c_q&=\max_{\substack{1\le i\le n\\i\ne q}}\left(1-p_{iq,q}^{(m)}\right),
&d_q&=\min_{\substack{1\le k\le n\\k\ne q}}p_{qk,q}^{(m)}.
\end{aligned}
\]
When $c_q+d_q>0$ define
\[
\varrho_q=\frac{a_q+b_q}{c_q+d_q},\qquad
\gamma_q=\frac{b_qd_q-a_qc_q}{c_q+d_q},
\]
and, for $n<\nu$, define
\[
\varepsilon_q=\min_{n<l\le\nu}p_{ql,q}^{(m)}.
\]
Denote
\[
\Upsilon_q=\{(x,y)\in\Om:x_q=y_q=0\},
\]
and let $\cA(E_q)$ denote the set of initial states whose trajectories converge to $E_q$.

\begin{lemma}
Let $a_qc_q<b_qd_q$. If $n<\nu$ assume in addition $\varepsilon_q>0$. Then
\[
\Phi_q(x,y)=(1-x_q)+\varrho_q(1-y_q)
\]
is a Lyapunov function relative to $E_q$. The equilibrium $E_q$ is locally exponentially asymptotically stable, and
\[
\cA(E_q)=\Om\setminus\Upsilon_q.
\]
Since $\codim_{\Om}\Upsilon_q=2$, the equilibrium $E_q$ is almost globally asymptotically stable: its basin of attraction $\cA(E_q)$ has full measure in $\Om$ with respect to Lebesgue measure on the affine hull of $\Om$.
\end{lemma}

\begin{proof}
Put
\[
u=1-x_q,\qquad
v=\sum_{\substack{l=1\\l\ne q}}^n y_l,\qquad
w=\sum_{l=n+1}^{\nu}y_l,
\]
where, when $n=\nu$, $w=0$. Then $\Phi_q=0$ only at $E_q$.

From system \eqref{eq:5}, using the stochastic and Volterra conditions, we obtain
\[
\dot u\le a_qx_qv-b_qy_qu.
\]
For $n<\nu$,
\[
\dot w\le-\varepsilon_qx_qw-T_q,
\qquad
\dot v\le c_qy_qu-d_qx_qv+T_q,
\]
where
\[
T_q=\sum_{\substack{i=1\\i\ne q}}^n\sum_{l=n+1}^{\nu}
\left(1-p_{il,l}^{(m)}\right)x_i y_l.
\]
For $n=\nu$ the term $T_q$ is absent from the estimate for $\dot v$.

Since
\[
b_q-\varrho_qc_q=\varrho_qd_q-a_q=\gamma_q,
\]
we have
\[
\dot\Phi_q\le-\gamma_qy_qu-\gamma_qx_qv-\varrho_q\varepsilon_qx_qw\le0,
\]
where, for $n=\nu$, the last term is absent.

Choose a neighborhood of $E_q$, in which $x_q\ge\delta$ and $y_q\ge\delta$ for some $\delta>0$, and then a sufficiently small $\varepsilon>0$ so that the set $\{\Phi_q\le\varepsilon\}$ is contained in this neighborhood. This set is positively invariant, and on it
\[
\dot\Phi_q\le-\gamma_0\Phi_q,
\]
where
\[
\gamma_0=\delta\min\left\{\gamma_q,\frac{\gamma_q}{\varrho_q},\varepsilon_q\right\}\quad(n<\nu),
\]
and at $n=\nu$,
\[
\gamma_0=\delta\min\left\{\gamma_q,\frac{\gamma_q}{\varrho_q}\right\}.
\]
Since $\Phi_q$ is equivalent to the distance to $E_q$ on $\Om$, this estimate proves local exponential stability.

From the global estimate for $\dot\Phi_q$ it follows that $\dot\Phi_q<0$ outside $\{E_q\}\cup\Upsilon_q$. On $\Upsilon_q$,
\[
\Phi_q=1+\varrho_q,
\]
and $\Upsilon_q$ is positively invariant.

If $(x^0,y^0)\notin\Upsilon_q$, then
\[
\Phi_q(x^0,y^0)<1+\varrho_q,
\]
so the monotonicity of $\Phi_q$ rules out any $\omega$-limit set with $\Upsilon_q$. By LaSalle's invariance principle,
\[
(x,y)\to E_q\qquad\text{as }t\to+\infty.
\]
Hence $\cA(E_q)=\Om\setminus\Upsilon_q$.
\end{proof}

As an example, let us apply Lemma~4 to the continuous analogue of the ninth extreme operator $W_9$ \cite{ref10}:
\[
\begin{cases}
\dot x_1=y_1-x_1,\\
\dot x_2=y_2-x_2,\\
\dot y_1=x_1y_2,\\
\dot y_2=x_2y_2-y_2.
\end{cases}
\]
Its equilibria are $E_1=((1,0),(1,0))$, $E_2=((0,1),(0,1))$. For $q=1$,
\[
a_1=b_1=d_1=1,\qquad c_1=0,
\]
so
\[
\varrho_1=2,\qquad \gamma_1=1,\qquad \Phi_1=x_2+2y_2.
\]
Here
\[
\Upsilon_1=\{E_2\},\qquad \cA(E_1)=\Om\setminus\Upsilon_1.
\]
Thus $E_1$ is locally exponentially and almost globally asymptotically stable. For $E_2$ the condition of Lemma~4 fails; analysis of the reduced Jacobian shows that $E_2$ is unstable.

The conditions of Lemma~4 are convenient for a global estimate. The next result refines the local stability criterion.

For any $q,i\in\{1,\ldots,n\}$, $q\ne i$, denote the two-type face
\[
\Om_{qi}=\{(x,y)\in\Om:x_j=0\ (j\notin\{q,i\}),\ y_l=0\ (l\notin\{q,i\})\}.
\]
Put
\[
\Delta_{qi}=p_{iq,q}^{(f)}p_{qi,q}^{(m)}-
\left(1-p_{qi,q}^{(f)}\right)\left(1-p_{iq,q}^{(m)}\right).
\]

\begin{lemma}
The equilibrium $E_q$ is locally exponentially asymptotically stable relative to $\Om$ if and only if
\[
\Delta_{qi}>0\qquad\text{for all }i\in\{1,\ldots,n\}\setminus\{q\},
\]
and, if $n<\nu$, $\varepsilon_q>0$.

If $\Delta_{qi}<0$ for at least one $i\ne q$, then $E_q$ is unstable. If $\Delta_{qi}=0$ for at least one $i\ne q$ or if $n<\nu$, $\varepsilon_q=0$, then $E_q$ is non-isolated and is not asymptotically stable.
\end{lemma}

\begin{proof}
After eliminating the dependent coordinates $x_q,y_q$ the reduced Jacobian of system \eqref{eq:5} has two-dimensional blocks
\[
J_{qi}^{(5)}=
\begin{pmatrix}
-A_{qi}&B_{qi}\\
C_{qi}&-D_{qi}
\end{pmatrix},
\]
where, for $i\ne q$,
\[
A_{qi}=p_{iq,q}^{(f)},\qquad
B_{qi}=1-p_{qi,q}^{(f)},\qquad
C_{qi}=1-p_{iq,q}^{(m)},\qquad
D_{qi}=p_{qi,q}^{(m)},
\]
with $0\le A_{qi},B_{qi},C_{qi},D_{qi}\le1$. For $n<\nu$ there are additional one-dimensional blocks
\[
-p_{ql,q}^{(m)},\qquad l>n.
\]
The block structure of the reduced Jacobian together with the Hurwitz criterion ensures local exponential asymptotic stability whenever the conditions of the first part of the lemma hold. If $\Delta_{qi}<0$, the corresponding block has a positive eigenvalue with a positive eigenvector, so small perturbations along that eigenvector remain within the invariant face $\Om_{qi}$, proving instability of $E_q$.

At $\Delta_{qi}=0$ the equilibrium equations on $\Om_{qi}$ written in the variables $x_i/x_q$, $y_i/y_q$ form a homogeneous linear system with a nonzero nonnegative solution. Consequently, at $\Delta_{qi}=0$ there exist equilibria in $\Om$ that accumulate at $E_q$. Food $n<\nu$ and $\varepsilon_q=0$ the segment of states with a single female type $q$ and male types $q,l$, where $l>n$ and $p_{ql,q}^{(m)}=0$, consists entirely of equilibria. In both cases $E_q$ is non-isolated, so asymptotic stability is impossible.
\end{proof}

The conditions of Lemma~4, namely, for each $i\ne q$,
\[
\left(1-p_{qi,q}^{(f)}\right)\left(1-p_{iq,q}^{(m)}\right)
\le a_qc_q<b_qd_q\le p_{iq,q}^{(f)}p_{qi,q}^{(m)},
\]
imply $\Delta_{qi}>0$. For $n<\nu$ the condition $\varepsilon_q>0$ coincides in both lemmas. Thus Lemma~5 gives an exact local criterion, while Lemma~4, under stronger sufficient conditions, additionally establishes
\[
\cA(E_q)=\Om\setminus\Upsilon_q.
\]

\subsection{Periodic trajectories on the diagonal invariant subset}

The stability results obtained above do not answer the question of the existence of periodic trajectories. Let us consider a corresponding special case.

Let $n=\nu=N_0$, $I_1\subseteq\{1,\ldots,N_0\}$, $|I_1|\ge3$. For $i\ne j$, $i,j\in I_1$ put
\[
a_{ij}^{(f)}=p_{ij,i}^{(f)}+p_{ji,i}^{(f)}-1,\qquad
a_{ij}^{(m)}=p_{ij,i}^{(m)}+p_{ji,i}^{(m)}-1,
\]
and $a_{ii}^{(f)}=0$, $a_{ii}^{(m)}=0$. Denote
\[
A_1^{(f)}=\left(a_{ij}^{(f)}\right)_{i,j\in I_1},\qquad
A_1^{(m)}=\left(a_{ij}^{(m)}\right)_{i,j\in I_1}.
\]
These matrices are skew-symmetric. Introduce
\[
\Delta=\left\{s=(s_i)_{i\in I_1}:s_i>0,\ \sum_{i\in I_1}s_i=1\right\},
\]
\[
\Om_1=\{(x,y):x_i=y_i=s_i,\ i\in I_1,\ s\in\Delta,\ x_q=y_q=0,\ q\notin I_1\}.
\]

\begin{lemma}
Let $A_1^{(f)}=A_1^{(m)}=:A_1$, $\rank A_1=2$, and suppose there exists
\[
C=(C_i)_{i\in I_1}\in\ker A_1,\qquad C_i>0.
\]
Then $\Om_1$ is invariant under system \eqref{eq:5}, and on $\Om_1$ the system takes the form
\begin{equation}
\label{eq:7}
\dot s_i=s_i(A_1s)_i,\qquad i\in I_1.
\end{equation}
Every trajectory with $s^0\in\Delta$, $A_1s^0\ne0$ is non-constant and periodic.
\end{lemma}

\begin{proof}
The Volterra condition preserves the vanishing coordinates outside $I_1$, while the equality $A_1^{(f)}=A_1^{(m)}=A_1$ ensures the invariance of the diagonal $x_i=y_i=s_i$. Hence \eqref{eq:7} holds on $\Om_1$.

The skew-symmetry of $A_1$ preserves $\sum_{i\in I_1}s_i=1$, and the equations preserve the positivity of the $s_i$; hence, $\Om_1$ is invariant.

Since $A_1$ is skew-symmetric with $\rank A_1=2$, there exist linearly independent $u,v$, for which
\[
A_1=uv^T-vu^T.
\]
Then
\[
s_i(t)=s_i^0e^{u_i\alpha(t)+v_i\beta(t)},
\]
where
\[
\dot\alpha(t)=Z_\beta(\alpha(t),\beta(t)),\qquad
\dot\beta(t)=-Z_\alpha(\alpha(t),\beta(t)),\qquad
\alpha(0)=\beta(0)=0,
\]
and
\[
Z(\alpha,\beta)=\sum_{i\in I_1}s_i^0e^{u_i\alpha+v_i\beta}.
\]
Hence $Z(\alpha(t),\beta(t))\equiv1$.

From $C\in\ker A_1$, $C_i>0$ it follows that $0$ lies in the interior of the convex hull of the vectors $(u_i,v_i)$. Hence $Z$ is strictly convex and coercive. From $A_1s^0\ne0$ it follows that $\nabla Z(0,0)\ne0$, so $\min Z<1$. Therefore the level set $Z=1$ is a regular simple closed curve. The Hamiltonian vector field does not vanish on it, so motion along this level set is periodic. Consequently, the trajectory $s(t)$ is periodic, and the condition $A_1s^0\ne0$ rules out a constant solution.
\end{proof}

Lemma~6 allows one, for each $N_0\ge3$, to construct an explicit one-parameter family of Volterra quadratic stochastic operators for two-sex populations (VQSOBP) possessing the stated periodicity property.

\begin{corollary}
For any $N_0\ge3$ and $0<\theta\le1$ and any $I_1\subseteq\{1,\ldots,N_0\}$, there exists a VQSOBP with $n=\nu=N_0$, such that all interior non-equilibrium trajectories of its continuous analogue on $\Om_1$ are periodic.
\end{corollary}

\begin{proof}
For $\bar q\in\{f,m\}$ and $1\le i<j\le N_0$ put
\[
p_{ij,i}^{(\bar q)}=p_{ji,i}^{(\bar q)}
=\frac12\left(1+\theta\sin\frac{2\pi(j-i)}{N_0}\right),
\]
\[
p_{ij,j}^{(\bar q)}=p_{ji,j}^{(\bar q)}
=\frac12\left(1-\theta\sin\frac{2\pi(j-i)}{N_0}\right).
\]
Put the diagonal coefficients equal to one and all others equal to zero. A direct check shows that these coefficients satisfy the stochastic and Volterra conditions, and $A_1^{(f)}=A_1^{(m)}=A_1$. The matrix $A_1$ is skew-symmetric, has rank $2$, and
\[
(1,1,\ldots,1)^T\in\ker A_1.
\]
Hence all the conditions of Lemma~6 are satisfied, and the assertion follows from it.
\end{proof}

\begin{remark}
Systems of the form \eqref{eq:7} with a skew-symmetric matrix $A_1$ have previously been studied, in particular in \cite{ref41}, where, under corresponding conditions on $\ker A_1$, periodicity of interior non-equilibrium trajectories was established. The corollary above provides an explicit realization of such a mechanism within the continuous analogue of the VQSOBP for every $N_0$.
\end{remark}

\begin{example}
Let $n=\nu=3$. For each $\bar q\in\{f,m\}$ and each cyclic pair $(i,j)=(1,2),(2,3),(3,1)$ put
\[
p_{ij,i}^{(\bar q)}=p_{ji,i}^{(\bar q)}=\frac34,\qquad
p_{ij,j}^{(\bar q)}=p_{ji,j}^{(\bar q)}=\frac14.
\]
For pairs of parents with identical indices, the heredity coefficient of the matching type is set to one, and all remaining coefficients not specified above are set to zero. Then
\[
A_1=\frac12
\begin{pmatrix}
0&1&-1\\
-1&0&1\\
1&-1&0
\end{pmatrix},
\qquad \rank A_1=2,\qquad
\ker A_1=\spann\{(1,1,1)\}.
\]
For
\[
x(0)=y(0)=s^0=\left(\frac12,\frac13,\frac16\right)
\]
we have $A_1s^0\ne0$, and on $\Om_1$ system \eqref{eq:5} takes the form
\[
\begin{cases}
\dot s_1=s_1(s_2-s_3)/2,\\
\dot s_2=s_2(s_3-s_1)/2,\\
\dot s_3=s_3(s_1-s_2)/2.
\end{cases}
\]
By Lemma~6 the corresponding trajectory is periodic. A numerical computation of its period gives $T\approx23.17873$. Figure~\ref{fig:periodic} shows projections of this trajectory onto the planes $(s_1,s_2)$ and $(s_2,s_3)$.

\begin{figure}[htbp]
\centering
\includegraphics[width=0.94\textwidth]{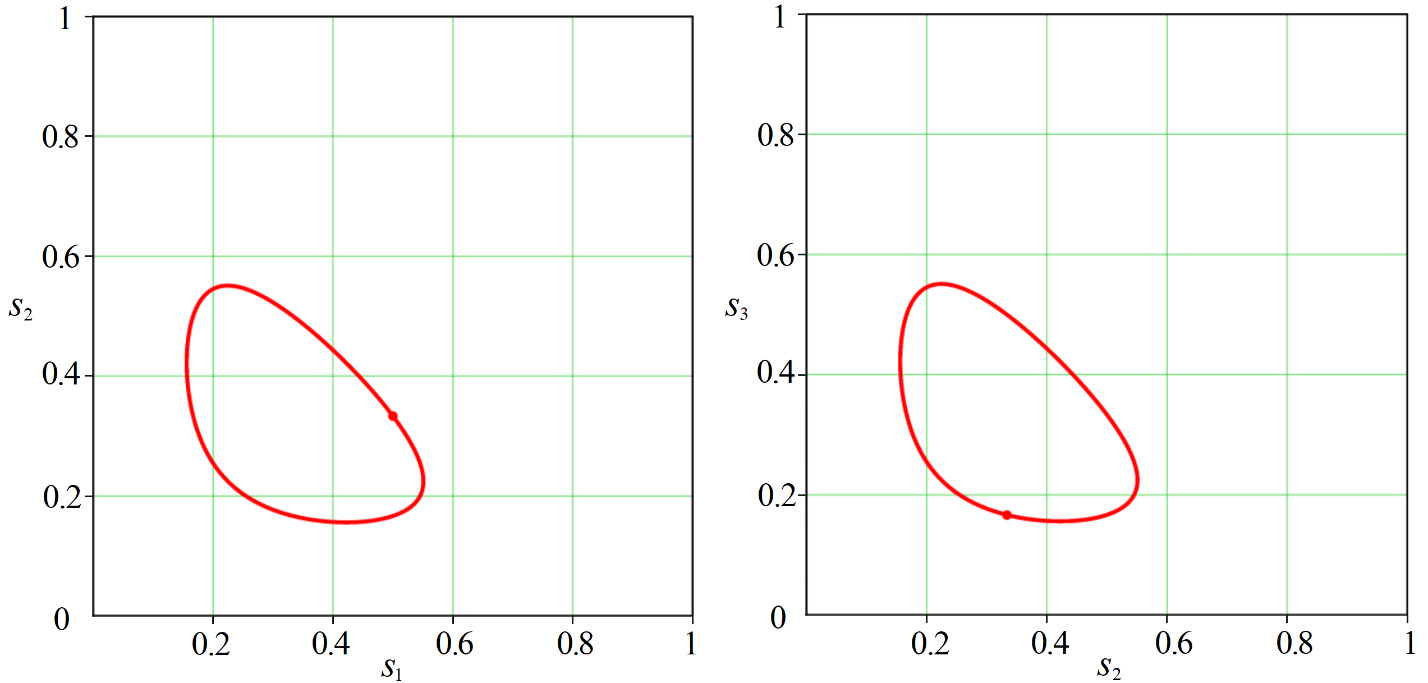}
\caption{Projections of the periodic trajectory onto the planes $(s_1,s_2)$ and $(s_2,s_3)$.}
\label{fig:periodic}
\end{figure}
\end{example}

\subsection{Comparison of discrete and continuous dynamics}

Let us compare the properties of system \eqref{eq:5} established above with the dynamics of the operator $W$.

By $\omega_{\mathrm{disk}}(z^0)$ and $\omega_{\mathrm{cont}}(z^0)$ we denote the $\omega$-limit sets of the corresponding trajectories, considered in $\Om$, and by $\cA_{\mathrm{disk}}(E_q)$ and $\cA_{\mathrm{cont}}(E_q)$ the sets of initial states whose trajectories converge to $E_q$. Thus,
\[
\omega_{\mathrm{cont}}(z^0)=\omega(z^0),\qquad
\cA_{\mathrm{cont}}(E_q)=\cA(E_q).
\]
Hyperbolicity is understood with respect to the reduced Jacobian matrices obtained after eliminating the dependent coordinates $x_q,y_q$.

\begin{lemma}
For a fixed $q\in\{1,\ldots,n\}$ the following statements hold:
\begin{enumerate}[label=$\arabic*^\circ$.,leftmargin=2.4em]
\item The point $E_q$ is simultaneously hyperbolic for $W$ and for system \eqref{eq:5} if and only if $\Delta_{qi}\ne0$ for all $i\ne q$, and, when $n<\nu$, $\varepsilon_q>0$.
\item The point $E_q$ is locally exponentially asymptotically stable relative to $\Om$ both for $W$, and for system \eqref{eq:5} if and only if $\Delta_{qi}>0$ for all $i\ne q$, and, when $n<\nu$, $\varepsilon_q>0$.

If $\Delta_{qi}<0$ for at least one $i\ne q$, then $E_q$ is unstable for $W$ and for system \eqref{eq:5}, regardless of its hyperbolicity.
\item If $\Delta_{qi}=0$ for at least one $i\ne q$ or if $n<\nu$, $\varepsilon_q=0$, and then $E_q$ s non-hyperbolic for $W$ and for system \eqref{eq:5}, non-isolated, and is not asymptotically stable.
\end{enumerate}
\end{lemma}

\begin{proof}
For fixed $i\ne q$ when $\Delta_{qi}>0$ the characteristic equation of the block $J_{qi}^{(5)}$ has two real negative roots $\lambda$, and the corresponding eigenvalues of the block
\[
J_{qi}^{(4)}=I_2+J_{qi}^{(5)},
\]
where $I_2$ is the identity matrix of order two, equal $\mu=1+\lambda$.

All roots satisfy $\lambda\ge-2$, with equality possible only when $A_{qi}=B_{qi}=C_{qi}=D_{qi}=1$, i.e., when $\Delta_{qi}=0$. Hence, when $\Delta_{qi}>0$, both roots satisfy
\[
-2<\lambda_{1,2}<0,\qquad -1<\mu_{1,2}=1+\lambda_{1,2}<1.
\]
If $\Delta_{qi}<0$, one of the roots $\lambda$ is positive, so $\mu>1$. If $\Delta_{qi}=0$ then
\[
\lambda_1=0,\qquad \mu_1=1,
\]
and the second root equals
\[
\lambda_2=-(A_{qi}+D_{qi}),\qquad \mu_2=1-A_{qi}-D_{qi}.
\]
For the additional coordinates $y_l$, $l>n$, the eigenvalues of the one-dimensional blocks of the reduced Jacobian matrices of system \eqref{eq:5} and of $W$ equal, respectively,
\[
-p_{ql,q}^{(m)},\qquad 1-p_{ql,q}^{(m)}.
\]
This proves the hyperbolicity criterion. When $\Delta_{qi}>0$ for all $i\ne q$ and, if $n<\nu$, $\varepsilon_q>0$, then, by the theorem on stability by first approximation, the point $E_q$ is locally exponentially asymptotically stable both for system \eqref{eq:5} and for $W$.

If $\Delta_{qi}<0$, the eigenvector of the block $J_{qi}^{(5)}$ corresponding to the positive eigenvalue $\lambda$, can be chosen with positive coordinates (see Lemma~5). The corresponding direction is admissible on the invariant two-type face $\Om_{qi}$, and for the operator $W$ this same eigenvector corresponds to an eigenvalue $\mu>1$. Hence $E_q$ is unstable relative to $\Om$ both for $W$, and for system \eqref{eq:5}.

If $\Delta_{qi}=0$ for at least one $i\ne q$ or, when $n<\nu$, $\varepsilon_q=0$, then by Lemma~5 the point $E_q$ is non-isolated and is not asymptotically stable. This proves $3^\circ$. Taking into account the case $\Delta_{qi}<0$ we obtain the necessity of the conditions $\Delta_{qi}>0$ for all $i\ne q$ and, when $n<\nu$, $\varepsilon_q>0$ in $2^\circ$.
\end{proof}

\begin{remark}
In the non-hyperbolic case, when $\Delta_{qi}\ge0$ for all $i\ne q$ Lemma 7 Lemma~7 does not allow one to judge the Lyapunov stability of $E_q$ or its preservation under the transition from $W$ to system \eqref{eq:5}.
\end{remark}

Let us turn to the conditions under which the $\omega$-limit sets of the two trajectories coincide.

\begin{lemma}
\begin{enumerate}[label=$\arabic*^\circ$.,leftmargin=2.4em]
\item Let, for a fixed $q\in\{1,\ldots,n\}$,
\[
\left(1-p_{qi,q}^{(f)}\right)\left(1-p_{jq,q}^{(m)}\right)
<p_{jq,q}^{(f)}p_{qi,q}^{(m)},
\qquad i,j\in\{1,\ldots,n\}\setminus\{q\},
\]
where $i=j$ is allowed, and when $n<\nu$, $\varepsilon_q>0$. Then
\[
\cA_{\mathrm{disk}}(E_q)=\cA_{\mathrm{cont}}(E_q)=\Om\setminus\Upsilon_q,
\]
with
\[
\omega_{\mathrm{disk}}(z^0)=\omega_{\mathrm{cont}}(z^0)=\{E_q\}
\qquad\text{for all }z^0\in\Om\setminus\Upsilon_q.
\]
\item For any $q,i\in\{1,\ldots,n\}$, $q\ne i$, the equality
\[
\omega_{\mathrm{disk}}(z^0)=\omega_{\mathrm{cont}}(z^0)
\]
holds for all $z^0\in\Om_{qi}$ if and only if
\[
\left(p_{iq,q}^{(f)},p_{qi,q}^{(f)},p_{iq,q}^{(m)},p_{qi,q}^{(m)}\right)\ne(1,0,0,1).
\]
When this condition holds, both trajectories with a common initial state converge to the same equilibrium. If $\Delta_{qi}>0$, the common limit equals $E_q$ for all $z^0\in\Om_{qi}\setminus\{E_i\}$; if $\Delta_{qi}<0$, it equals $E_i$ for all $z^0\in\Om_{qi}\setminus\{E_q\}$. For the excluded coefficient set, the equality $\omega_{\mathrm{disk}}(z^0)=\omega_{\mathrm{cont}}(z^0)$ holds precisely when $x_i^0=y_i^0$.
\end{enumerate}
\end{lemma}

\begin{proof}
By virtue of the conditions o $1^\circ$ for some $\rho>0$ the function
\[
L_q=(1-x_q)+\rho(1-y_q)
\]
does not increase for $W$ and system \eqref{eq:5}, and strictly decreases outside $\{E_q\}\cup\Upsilon_q$. Since its maximum is attained exactly on $\Upsilon_q$, for $z^0\notin\Upsilon_q$ the monotonicity excludes limit points on this face. By the invariance principle for $W$ and for system \eqref{eq:5}, both $\omega$-limit sets equal $\{E_q\}$. Since $\Upsilon_q$ is closed, invariant, and does not contain $E_q$, we obtain the stated equality of the basins of attraction.

On the two-type face $\Om_{qi}$, the condition of $1^\circ$ is equivalent to $\Delta_{qi}>0$, and $\Upsilon_q\cap\Om_{qi}=\{E_i\}$. Hence, when $\Delta_{qi}>0$ both trajectories converge to $E_q$ for all initial states except $E_i$. If $\Delta_{qi}<0$, then $\Delta_{iq}=-\Delta_{qi}>0$, and both trajectories converge to $E_i$ for all initial states except $E_q$.

Let $\Delta_{qi}=0$ and
\[
\left(p_{iq,q}^{(f)},p_{qi,q}^{(f)},p_{iq,q}^{(m)},p_{qi,q}^{(m)}\right)\ne(1,0,0,1).
\]
In cases where one of the coordinates is constant, the equation for the other coordinate becomes linear, and the corresponding discrete map becomes affine; the coefficient of the varying coordinate in the discrete map lies in the interval $[0,1]$; if it equals unity, both trajectories are stationary. Therefore, the continuous and discrete trajectories share the same limit.

Let
\[
p_{iq,q}^{(f)}=0,\qquad p_{iq,q}^{(m)}=1,\qquad
p_{qi,q}^{(f)}<1,\qquad p_{qi,q}^{(m)}>0,
\]
then the function
\[
H_1=p_{qi,q}^{(m)}x_i+\left(1-p_{qi,q}^{(f)}\right)y_i
\]
is a common first integral for $W$ and for system \eqref{eq:5}. The coordinate $x_i$ is non-decreasing, $y_i$ is non-increasing, and on each level segment of $H_1$ the equilibrium is unique. Hence both trajectories converge to the same equilibrium.

In the remaining cases $A_{qi},C_{qi}>0$, and the common first integral
\[
H_2=\left(1-p_{iq,q}^{(m)}\right)x_i+p_{iq,q}^{(f)}y_i
\]
defines invariant segments, on each of which the equilibrium is unique. After eliminating one coordinate using the condition $H_2=\const$ the discrete dynamics reduces to a one-dimensional map $T$, for which $|T'(u)|<1$ at interior points of the segment. Hence the distance to the equilibrium strictly decreases, and, by compactness of the segment, the discrete trajectory converges to it. The continuous trajectory on the same segment tends monotonically to this equilibrium. Single-point levels consist of equilibria. Consequently, the continuous and discrete trajectories converge to the same equilibrium.

For the excluded coefficient set, the discrete trajectory on $\Om_{qi}$ forms a two-cycle when $x_i^0\ne y_i^0$, while the continuous trajectory converges to the equilibrium
\[
x_i=y_i=\frac{x_i^0+y_i^0}{2}.
\]
Consequently, in this case $\omega_{\mathrm{disk}}(z^0)=\omega_{\mathrm{cont}}(z^0)$ if and only if $x_i^0=y_i^0$.
\end{proof}

Periodic regimes require a separate comparison.

\begin{lemma}
\begin{enumerate}[label=$\arabic*^\circ$.,leftmargin=2.4em]
\item If $\Gamma$ is a non-constant periodic orbit of system \eqref{eq:5}, then $W(\Gamma)\nsubseteq\Gamma$.
\item On $\Om_{qi}$ there are no non-constant periodic orbits of system \eqref{eq:5}. Non-constant cycles of $W$ exist precisely for the coefficient set excluded in $2^\circ$ of Lemma~8. All of them have minimal period $2$.
\item Let $A_1^{(f)}=A_1^{(m)}=:A_1$. Then on $\Om_1$ the operator $W$ has no non-constant cycles. For the existence of an interior non-constant periodic trajectory of system \eqref{eq:5} on $\Om_1$ it is necessary that
\[
\ker A_1\cap(0,\infty)^{|I_1|}\ne\varnothing.
\]
If $\rank A_1=2$, this condition is also sufficient, and all interior non-equilibrium trajectories of system \eqref{eq:5} on $\Om_1$ are periodic.
\end{enumerate}
\end{lemma}

\begin{proof}
At the point $z^*\in\Gamma$, where the Euclidean norm attains its maximum, we have
\[
\langle z^*,F(z^*)\rangle=0,\qquad F(z^*)\ne0,
\]
where $F$ is the right-hand side of system \eqref{eq:5}. Hence
\[
\|W(z^*)\|^2=\|z^*\|^2+\|F(z^*)\|^2>\|z^*\|^2,
\]
which proves $1^\circ$.

Item $2^\circ$ follows directly from item $2^\circ$ of Lemma~8.

Invariance of $\Om_1$ for $W$ and for system \eqref{eq:5} is equivalent to the equality $A_1^{(f)}=A_1^{(m)}=A_1$.

Suppose a discrete cycle exists. Taking the logarithm of the coordinate equalities and summing them over the period, using the inequality $\ln(1+\xi)\le\xi$ we obtain $A_1\bar s\ge0$, where $\bar s>0$ is the average over the cycle. Since $A_1$ is skew-symmetric, $\bar s^TA_1\bar s=0$, so $A_1\bar s=0$. Consequently, equality is attained in all the logarithmic inequalities used, and every point of the cycle is fixed.

For a continuous periodic trajectory, averaging the equations
\[
\frac{d}{dt}\ln s_i=(A_1s)_i
\]
over a period shows that its positive average belongs to $\ker A_1$. Therefore, the condition
\[
\ker A_1\cap(0,\infty)^{|I_1|}\ne\varnothing
\]
is necessary for any rank of $A_1$. When $\rank A_1=2$ the sufficiency of this condition and the periodicity of all non-equilibrium trajectories on $\Om_1$ follow from Lemma~6.
\end{proof}

\begin{remark}
Despite the coincidence of the fixed points of the operator $W$ and the equilibria of system \eqref{eq:5}, their periodic regimes generally differ. On a two-type face, a discrete two-cycle may fail to persist under the transition to the continuous analogue, whereas on $\Om_1$ system \eqref{eq:5} may possess non-constant periodic trajectories in the absence of non-constant cycles of the operator $W$.
\end{remark}

Let us now single out a class for which the $\omega$-limit sets of $W$ and of system \eqref{eq:5} coincide for every initial state.

Put, for each $z^0\in\Om$,
\[
q=\min\{i\in\{1,\ldots,n\}:x_i^0+y_i^0>0\}.
\]

\begin{theorem}
Let
\begin{equation}
\label{eq:8}
\left(1-p_{qi,q}^{(f)}\right)\left(1-p_{jq,q}^{(m)}\right)
<p_{jq,q}^{(f)}p_{qi,q}^{(m)},
\qquad 1\le q<n,\quad q<i,j\le n,
\end{equation}
where $i=j$ is allowed, and, when $n<\nu$, additionally
\begin{equation}
\label{eq:9}
\varepsilon_q>0,\qquad q=1,\ldots,n.
\end{equation}
Then
\[
\lim_{k\to\infty}W^k(z^0)=\lim_{t\to+\infty}z(t;z^0)=E_q,
\qquad
\omega_{\mathrm{disk}}(z^0)=\omega_{\mathrm{cont}}(z^0)=\{E_q\}.
\]
The fixed points of the operator $W$ and the equilibria of system \eqref{eq:5} are exhausted by $E_1,\ldots,E_n$, and, for each $q=1,\ldots,n$,
\[
\cA_{\mathrm{disk}}(E_q)=\cA_{\mathrm{cont}}(E_q)
=\{z\in\Om:x_i=y_i=0\ (i<q),\ x_q+y_q>0\}.
\]
All points $E_q$ are hyperbolic for $W$ and for system \eqref{eq:5}. The dimensions of the unstable and stable spectral subspaces of the reduced Jacobian matrices at $E_q$ equal, respectively, $q-1$ and $n+\nu-q-1$.

Relative to $\Om$ the point $E_1$ is locally exponentially asymptotically stable, while the points $E_q$, $q>1$, are unstable both for $W$ and for system \eqref{eq:5}. There are no non-constant cycles of the operator $W$ and no periodic orbits of system \eqref{eq:5} on
\end{theorem}

\begin{proof}
Fix $z^0\in\Om$ and the corresponding index $q$. By the Volterra condition, the face defined by the equalities $x_i=y_i=0$ for $i<q$, is invariant for $W$ and for system \eqref{eq:5}. For $q<n$, after eliminating the absent types, conditions \eqref{eq:8}, \eqref{eq:9} ensure that the conditions of $1^\circ$ of Lemma~8 hold for the restriction to this face. Since $x_q^0+y_q^0>0$, both $\omega$-limit sets equal $\{E_q\}\{E_q\}$; compactness of $\Om$ gives convergence of both trajectories to $E_q$.

For $q=n$ on the face under consideration $x_n=1$. If $n<\nu$, then for $l>n$,
\[
y_l^{(k+1)}=\left(1-p_{nl,n}^{(m)}\right)y_l^{(k)}
\qquad\text{and}\qquad
\dot y_l=-p_{nl,n}^{(m)}y_l.
\]
From \eqref{eq:9} it follows that $y_l\to0$ as $k\to+\infty$, whence $y_n\to1$. When $n=\nu$ this face consists of $E_n$.

Consequently, the formula for the common limit is established for all initial states. It determines the stated basins of attraction; its application to constant trajectories rules out other fixed points and equilibria. Convergence of all trajectories to equilibria rules out non-constant periodic regimes.

Let now $q\in\{1,\ldots,n\}$ be arbitrary. Applying \eqref{eq:8} with $i=j$ and using $\Delta_{qi}=-\Delta_{iq}$, we obtain
\[
\Delta_{qi}>0\quad(i>q),\qquad \Delta_{qi}<0\quad(i<q).
\]
Consequently, taking \eqref{eq:9} into account, by Lemma~7 all $E_q$ are hyperbolic for $W$ and for system \eqref{eq:5}; $E_1$ is locally exponentially asymptotically stable, while $E_q$, $q>1$, are unstable relative to $\Om$.

Each $i<q$ corresponds to exactly one unstable eigenvalue of each of the two reduced Jacobian matrices, while the remaining eigenvalues are stable. Hence the dimension of the unstable spectral subspace equals $q-1$. Since the order of each reduced matrix equals $n+\nu-2$, the dimension of the stable spectral subspace equals $n+\nu-q-1$.
\end{proof}

\begin{remark}
The class of VQSOBPs satisfying the conditions of Theorem~1 is non-empty. Choose $a_0,b_0\in(0,1)$ so that $a_0+b_0>1$, and, for every pair of parents with distinct indices $i,j\le n$ set the probability of inheriting the smaller index equal to $a_0$ for female offspring and $b_0$ for male offspring. If $n<\nu$ additionally put
\[
p_{ql,q}^{(m)}=c_0\in(0,1),\qquad 1\le q\le n<l\le\nu.
\]
The remaining coefficients are determined by the stochastic and Volterra conditions. Then all conditions of Theorem~1 are satisfied.
\end{remark}

\section{Analytical solution of the system}

Eliminate the coordinates
\[
x_n=1-\sum_{i=1}^{n-1}x_i,\qquad
y_\nu=1-\sum_{l=1}^{\nu-1}y_l
\]
and introduce
\[
u=(x_1,\ldots,x_{n-1},y_1,\ldots,y_{\nu-1})^T\in\R^N,
\qquad N=n+\nu-2.
\]
System \eqref{eq:5} then takes the form
\begin{equation}
\label{eq:10}
\dot u=F_1(u),\qquad u(0)=u^0,
\end{equation}
where $u^0$ is obtained from the initial state $z^0\in\Om$ by eliminating the coordinates $x_n,y_\nu$.

Here
\[
F_1(u)=\cC+L_0u+B(u,u),
\]
where $\cC\in\R^N$, $L_0$ is a linear operator, and $B$ is a symmetric bilinear mapping. Note that for special relations between the coefficients, system \eqref{eq:10} can be integrated explicitly (see Section~8).

Let $u^*$ be an equilibrium of system \eqref{eq:10}. Setting $v=u-u^*$ we obtain
\begin{equation}
\label{eq:11}
\dot v=Jv+B(v,v),\qquad v(0)=v^0,
\end{equation}
where
\[
Jv=F_1'(u^*)v=L_0v+2B(u^*,v),\qquad v^0=u^0-u^*.
\]
Fix a norm $\|\cdot\|$ on $\R^N$ and the corresponding induced matrix norm.

Assume $J$ is a Hurwitz matrix. Then there exist $M\ge1$ and $\omega_0>0$, such that
\[
\|e^{Jt}\|\le Me^{-\omega_0t},\qquad t\ge0.
\]
Put
\[
\cL[w]=\dot w-Jw,\qquad \cN[w]=\cL[w]-B(w,w),
\]
and choose
\[
v_0(t)=e^{Jt}v^0.
\]
For a continuous bounded function $\varphi:[0,+\infty)\to\R^N$ put
\[
\cH[\varphi](t)=\int_0^t e^{J(t-\tau)}\varphi(\tau)\,d\tau.
\]
Then
\begin{equation}
\label{eq:12}
\|\cH[\varphi]\|_\infty\le K\|\varphi\|_\infty,
\qquad K=\frac{M}{\omega_0},
\qquad \|\varphi\|_\infty=\sup_{t\ge0}\|f(\varphi)\|.
\end{equation}
Define the deformation equation in the form
\[
(1-p)\cL[V-v_0]=p\hbar\cN[V],\qquad V(0,p,\hbar)=v^0,
\]
where $p\in[0,1]$, $\hbar\in\R$. We seek the solution as a formal series
\begin{equation}
\label{eq:13}
V(t,p,\hbar)=v_0(t)+\sum_{k=1}^{\infty}v_k(t,\hbar)p^k,
\qquad v_k(0,\hbar)=0,\quad k\ge1.
\end{equation}
Comparing coefficients at $p^k$ gives
\[
\cL[v_k-\chi_kv_{k-1}]=\hbar R_k,
\qquad \chi_1=0,\qquad \chi_k=1\quad(k\ge2),
\]
where
\[
R_k=\cL[v_{k-1}]-\sum_{i=0}^{k-1}B(v_i,v_{k-1-i}).
\]
Hence
\[
v_1(t,\hbar)=-\hbar\int_0^t e^{J(t-\tau)}B(v_0(\tau),v_0(\tau))\,d\tau,
\]
\begin{equation}
\label{eq:14}
\begin{aligned}
v_k(t,\hbar)={}&(1+\hbar)v_{k-1}(t,\hbar)\\
&-\hbar\int_0^t e^{J(t-\tau)}
\sum_{i=0}^{k-1}B\bigl(v_i(\tau,\hbar),v_{k-1-i}(\tau,\hbar)\bigr)\,d\tau,
\qquad k\ge2.
\end{aligned}
\end{equation}
Define
\[
C_B=\max_{\|a\|=\|b\|=1}\|B(a,b)\|,
\qquad
M_k=\sup_{t\ge0}\|v_k(t,\hbar)\|,\quad k\ge1,
\]
and put
\[
\upsilon=|1+\hbar|,\qquad \mu=|\hbar|KC_B.
\]

\begin{theorem}
Let $J$ be a Hurwitz matrix, $-2<\hbar<0$ and
\begin{equation}
\label{eq:15}
4|\hbar|KC_BM_0\le1-|1+\hbar|,
\end{equation}
where
\[
M_0=\sup_{t\ge0}\|e^{Jt}v^0\|.
\]
Then the series
\begin{equation}
\label{eq:16}
u(t)=u^*+\sum_{k=0}^{\infty}v_k(t,\hbar)
\end{equation}
converges absolutely and uniformly on $[0,+\infty)$, and its sum is the unique solution of problem \eqref{eq:10}.
\end{theorem}

\begin{proof}
From \eqref{eq:12} and \eqref{eq:14} we have
\[
M_1\le\mu M_0^2,\qquad
M_k\le\upsilon M_{k-1}+\mu\sum_{i=0}^{k-1}M_iM_{k-1-i},\qquad k\ge2.
\]
For
\[
A_L=\sum_{k=0}^L M_k,\qquad L\ge1,
\]
we obtain
\begin{equation}
\label{eq:17}
A_L\le(1-\upsilon)M_0+\upsilon A_{L-1}+\mu A_{L-1}^2.
\end{equation}
For $\mu>0$ put
\[
V^*=\frac{1-\upsilon-\sqrt{(1-\upsilon)^2-4\mu(1-\upsilon)M_0}}{2\mu}.
\]
Condition \eqref{eq:15} ensures the reality of $V^*$, and
\[
M_0\le V^*,\qquad
(1-\upsilon)M_0+\upsilon V^*+\mu(V^*)^2=V^*.
\]
Since the right-hand side of \eqref{eq:17} is non-decreasing in $A_{L-1}$, induction gives $A_L\le V^*$ for all $L$. For $\mu=0$ all $v_k$, $k\ge1$, equal to zero. Thus,
\[
\sum_{k=0}^{\infty}M_k<\infty,
\]
and series \eqref{eq:16} converges absolutely and uniformly.

Summing the recurrence relations \eqref{eq:14} over $k$ and using the absolute convergence of the series, we obtain an integral equation equivalent to problem \eqref{eq:11}. Hence the sum of series \eqref{eq:16} is a solution of problem \eqref{eq:10}. Uniqueness follows from the local Lipschitz continuity of the polynomial mapping $F_1$.
\end{proof}

\begin{remark}
For $-2<\hbar<0$ condition \eqref{eq:15} is equivalent to
\[
4KC_BM_0\le g(\hbar),\qquad
g(\hbar)=
\begin{cases}
-\dfrac{2+\hbar}{\hbar},&-2<\hbar<-1,\\[1ex]
1,&-1\le\hbar<0.
\end{cases}
\]
Hence values of $\hbar$ from the interval $[-1,0)$ give, within this criterion, the least restrictive sufficient condition
\[
4KC_BM_0\le1.
\]
\end{remark}

For illustration, let us present the zeroth and first terms $v_0(t)$ and $v_1(t,\hbar)$ of the homotopy series \eqref{eq:13}. Let
\[
v^0=\sum_{i=1}^N a_{0i}e_i,
\qquad
B(e_i,e_j)=\sum_{s=1}^N b_{ij}^{s}e_s,
\]
where $a_{0i}$ is the coordinate of the initial vector in the Jordan basis; $e_i$ is a vector of the Jordan basis; $h_i$ is the depth of the vector $e_i$ in its chain; we set $e_{i-q}=0$, if the index $i-q$ falls outside this chain.

Then
\[
v_0(t)=e^{Jt}v^0
=\sum_{i=1}^N\sum_{q=0}^{h_i}a_{0i}e^{\lambda_it}\frac{t^q}{q!}e_{i-q},
\]
where $\lambda_i$ is the eigenvalue corresponding to the Jordan chain containing $e_i$; the depth of the initial eigenvector of the chain is taken to equal zero.

The first term has the form
\begin{align*}
v_1(t,\hbar)
={}&-\hbar
\sum_{i=1}^N\sum_{q=0}^{h_i}
\sum_{j=1}^N\sum_{r=0}^{h_j}
\sum_{s=1}^N\sum_{m=0}^{h_s}
\sum_{l=0}^{m}
\frac{a_{0i}a_{0j}b_{i-q,j-r}^{s}}{q!\,r!\,m!}
\binom{m}{l}(-1)^l t^{m-l}\varkappa!\\
&\hspace{25mm}\times e^{\lambda_st}I_{\varkappa}(\widetilde\lambda,t)e_{s-m},
\end{align*}
where
\[
\varkappa=q+r+l,\qquad \widetilde\lambda=\lambda_s-\lambda_i-\lambda_j,
\]
and
\[
I_{\varkappa}(\widetilde\lambda,t)=
\begin{cases}
\displaystyle
\frac{1}{\widetilde\lambda^{\varkappa+1}}
\left[1-e^{-\widetilde\lambda t}
\sum_{l=0}^{\varkappa}\frac{(\widetilde\lambda t)^l}{l!}\right],
&\widetilde\lambda\ne0,\\[3ex]
\displaystyle
\frac{t^{\varkappa+1}}{(\varkappa+1)!},
&\widetilde\lambda=0.
\end{cases}
\]

\begin{lemma}
Let $J$ be a Hurwitz matrix with Jordan data $\{\lambda_i,h_i\}_{i=1}^N$. Then for every $k\ge0$ the term $v_k(t,\hbar)$ of series \eqref{eq:16} has the form
\[
v_k(t,\hbar)=\sum_{s=1}^N\sum_{\lambda\in S_k}\sum_{l=0}^{P_k}
C_{s,l,\lambda}^{(k)}t^le^{\lambda t}e_s,
\]
where
\[
S_k=\{\lambda_{i_1}+\cdots+\lambda_{i_r}:1\le r\le k+1\},
\qquad i_1,\ldots,i_r\in\{1,\ldots,N\},
\]
and
\[
P_k\le h_{\max}+k(2h_{\max}+1),\qquad
h_{\max}=\max_{1\le i\le N}h_i.
\]
The coefficients $C_{s,l,\lambda}^{(k)}$ pare determined recurrently by the quantities $a_{0i},\lambda_i,h_i,b_{ij}^{s}$ and $\hbar$.

If the matrix $J$ is diagonalizable and its eigenvalues are real, negative, and linearly independent over $\mathbb{Q}$, then $C_{s,l,\lambda}^{(k)}=0$, for $l\ge1$, i.e., each $v_k(t,\hbar)$ is a finite sum of exponentials.
\end{lemma}

\begin{proof}
The assertion follows by induction on $k$ from \eqref{eq:14}. Under quadratic convolution, the exponents of the exponentials and the degrees of the polynomial factors add, while the action of the Duhamel operator augments the set of possible exponents by the eigenvalues of $J$ and increases the maximal degree of the polynomial factor by at most $h_{\max}+1$. This gives the stated $S_k$ and the estimate for $P_k$. In the diagonalizable case, linear independence of the eigenvalues over $\mathbb{Q}$ excludes resonances, so polynomial factors do not arise in subsequent terms either.
\end{proof}

\section{The general case of quadratic-linear first-order systems of ODEs in $\R^N$}

The applied HAM scheme extends to autonomous first-order systems of ODEs with quadratic right-hand sides, provided the convergence conditions established in Section~4 are satisfied.

Consider the Cauchy problem for the system
\begin{equation}
\label{eq:18}
\dot x_i=\sum_{j=1}^N a_{ij}x_j^2+
\sum_{1\le j<k\le N}b_{i,jk}x_jx_k+
\sum_{j=1}^N c_{ij}x_j+g_i,
\qquad x_i(0)=x_i^0,\quad i=1,\ldots,N,
\end{equation}
where here $N\ge1$ is arbitrary, $x=(x_1,\ldots,x_N)\in\R^N$, and $a_{ij},b_{i,jk},c_{ij},g_i\in\R$ are given coefficients.

Since system \eqref{eq:18} has the same quadratic structure as system \eqref{eq:10}, Theorem~2 and the corresponding HAM scheme apply to the Cauchy problem \eqref{eq:18} whenever the conditions of Theorem~2 are satisfied.

It is natural to compare the resulting HAM scheme with the construction of spectral power series in \cite{ref23}. In that work, for real negative eigenvalues linearly independent over $\mathbb{Q}$, an exponential expansion converging for $t>t_0$ was obtained. In the examples considered there, the free parameters are linked to the initial data through additional algebraic relations. In HAM, the initial conditions are accounted for directly in the construction of the series, and the criterion of Theorem~2 requires neither the reality of the eigenvalues nor absence of resonance. The spectral structure of the terms of the series is described by Lemma~10 and admits polynomial factors multiplying the exponentials.

\section{A special case of the VQSOBP ($n=\nu=2$)}

Let us set
\[
x=x_1,\qquad x_2=1-x,\qquad y=y_1,\qquad y_2=1-y.
\]
In the notation of Section~3.1 we have
\[
a_1=1-p_{12,1}^{(f)},\qquad
b_1=p_{21,1}^{(f)},\qquad
c_1=1-p_{21,1}^{(m)},\qquad
d_1=p_{12,1}^{(m)}.
\]
Then system \eqref{eq:5} on $\Om_2=S^1\times S^1$ takes the form
\begin{equation}
\label{eq:19}
\begin{cases}
\dot x=b_1y(1-x)-a_1x(1-y)=f_1(x,y),\\
\dot y=d_1x(1-y)-c_1y(1-x)=f_2(x,y),
\end{cases}
\end{equation}
where $a_1,b_1,c_1,d_1\in[0,1]$. The set $\Om_2$ is positively invariant.

The points $O=(0,0)$ and $E=(1,1)$ are equilibria. Here
\[
\det J(O)=a_1c_1-b_1d_1,\qquad
\det J(E)=b_1d_1-a_1c_1,
\]
and the discriminants of the characteristic equations equal
\[
D_O=(a_1-c_1)^2+4b_1d_1,\qquad
D_E=(b_1-d_1)^2+4a_1c_1.
\]
Consequently the eigenvalues are real: for $b_1d_1<a_1c_1$ the point $O$ is an asymptotically stable node, while $E$ is a saddle; for $b_1d_1>a_1c_1$ their roles are interchanged. At $O$ only the types $\cF_2,\cM_2$, are present, while at $E$ only $\cF_1,\cM_1$ are present. In degenerate cases additional boundary equilibria may arise, which are not considered further.

Further, assume $a_1,b_1,c_1,d_1>0$. Interior equilibria exist if and only if
\[
b_1d_1=a_1c_1.
\]
In this case the set of all equilibria has the form
\[
\Gamma_1=\left\{(x,y)\in\Om_2:
 y=\frac{a_1x}{b_1+(a_1-b_1)x}\right\}.
\]
The curve $\Gamma_1$ is strictly increasing and joins $O$ and $E$; the interior equilibria constitute $\Gamma_1\setminus\{O,E\}$.

For the Dulac function
\[
D(x,y)=\frac{1}{x(1-x)y(1-y)}
\]
at $(x,y)\in\operatorname{int}\Om_2$ we have
\[
\frac{\partial(Df_1)}{\partial x}+
\frac{\partial(Df_2)}{\partial y}<0,
\]
so that in the interior of $\Om_2$ there are no non-constant periodic orbits.

For $b_1d_1\ne a_1c_1$ consider the function
\[
H_3(x,y)=y+\frac{c_1}{b_1}x.
\]
By virtue of \eqref{eq:19} we obtain
\[
\dot H_3=\frac{b_1d_1-a_1c_1}{b_1}x(1-y).
\]
Consequently, for interior initial states, when $b_1d_1<a_1c_1$ the trajectory tends to the equilibrium $O=(0,0)$, and when $b_1d_1>a_1c_1$ -- to the equilibrium $E=(1,1)$ as $t\to+\infty$.

When $b_1d_1=a_1c_1$ the function $H_3$ is a first integral. Each of its level lines intersects $\Gamma_1$ at exactly one point, to which the corresponding interior trajectory converges monotonically. On $\Gamma_1$ one eigenvalue of the Jacobian matrix is zero and the second is negative; hence $\Gamma_1$ is a normally attracting curve of equilibria, whereas its individual points are not asymptotically stable (Fig.~\ref{fig:gamma}).

\begin{figure}[htbp]
\centering
\includegraphics[width=0.50\textwidth]{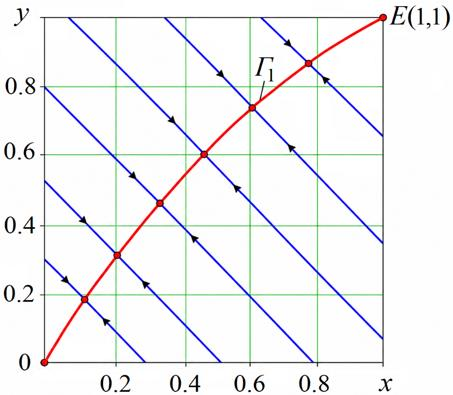}
\caption{Schematic representation of the curve of equilibria $\Gamma_1$ and trajectories on the level sets of $H_3$ when $b_1d_1=a_1c_1$.}
\label{fig:gamma}
\end{figure}

Discrete dynamical systems and their continuous analogues may exhibit different periodic regimes (see Remark~3). Consider the continuous analogue of $W_{10}$ \cite{ref10}, which in independent coordinates has the form
\[
W_{10}(x,y)=(y,x)
\]
and corresponds to $a_1=b_1=c_1=d_1=1$. For $x^0\ne y^0$ the discrete trajectory forms a cycle of minimal period two, whereas the continuous analogue
\[
\dot x=y-x,\qquad \dot y=x-y
\]
has the solution
\[
x(t)=\frac{x^0+y^0}{2}+\frac{x^0-y^0}{2}e^{-2t},
\qquad
y(t)=\frac{x^0+y^0}{2}-\frac{x^0-y^0}{2}e^{-2t},
\]
so that
\[
x(t),y(t)\to\frac{x^0+y^0}{2}\qquad\text{as }t\to+\infty.
\]
Thus the passage from $W$ to $F(z)=W(z)-z$, does not, in general, preserve periodic dynamics (see Section~3). Since $J_F=J_W-I_2$, to each eigenvalue $\mu_W$ of the matrix $J_W$ there corresponds an eigenvalue $\lambda_F=\mu_W-1$ of the matrix $J_F$.

For $W_{10}$ the value $\mu_W=-1$, associated with the two-cycle dynamics, corresponds to $\lambda_F=-2$, i.e., to an exponentially stable direction of the continuous flow. The second eigenvalue $\mu_W=1$ passes into $\lambda_F=0$, corresponding to the direction along the diagonal of equilibria. In particular,
\[
|\mu_W|<1\ \Longrightarrow\ \operatorname{Re}(\mu_W-1)<0,
\]
but the converse statement is not true.

\section{Illustrative examples}

\begin{example}
Consider the Cauchy problem for the continuous analogue of $W_9$ \cite{ref10}:
\[
\begin{cases}
\dot x_1=y_1-x_1,\\
\dot x_2=y_2-x_2,\\
\dot y_1=x_1y_2,\\
\dot y_2=x_2y_2-y_2,
\end{cases}
\qquad
(x_1(0),x_2(0),y_1(0),y_2(0))=(x_1^0,x_2^0,y_1^0,y_2^0)\in\Om_2,
\]
for which no solution in quadratures was obtained in \cite{ref26}.

Eliminating $x_2=1-x_1$ and $y_2=1-y_1$, we obtain
\begin{equation}
\label{eq:20}
\begin{cases}
\dot x_1=y_1-x_1,\\
\dot y_1=x_1(1-y_1),
\end{cases}
\qquad
(x_1(0),y_1(0))=(x_1^0,y_1^0).
\end{equation}
Its equilibria are $(0,0)$ -- a hyperbolic saddle, and $(1,1)$ -- an asymptotically stable improper node.

Set $u=x_1-1$, $v=y_1-1$. Then
\begin{equation}
\label{eq:21}
\begin{cases}
\dot u=v-u,\\
\dot v=-uv-v,
\end{cases}
\end{equation}
and, since $v=\dot u+u$,
\begin{equation}
\label{eq:22}
\ddot u+2\dot u+u\dot u+u+u^2=0,
\end{equation}
\begin{equation}
\label{eq:23}
u(0)=u_0,\qquad \dot u(0)=v_0-u_0.
\end{equation}
The stationary solutions are $u\equiv0$ and $u\equiv-1$. For $v_0=0$ the trajectory lies on the invariant line $v=0$, and
\[
u(t)=u_0e^{-t}.
\]

Consider now the more general problem
\begin{equation}
\label{eq:24}
\ddot u+2\dot u+u=f(u,\dot u),
\end{equation}
\begin{equation}
\label{eq:25}
u(0)=u_0,\qquad \dot u(0)=v_0-u_0,
\end{equation}
where it is assumed that for $|U|<r_1$, $|V|<r_2$ the function $f$ admits the absolutely convergent power series representation
\[
f(U,V)=\sum_{\mathfrak n,\mathfrak m\ge0}a_{\mathfrak n,\mathfrak m}U^{\mathfrak n}V^{\mathfrak m},
\qquad a_{\mathfrak n,\mathfrak m}\in\R,
\]
where $0<r_1,r_2\le+\infty$.

Set
\[
\cN_1[u]=\ddot u+2\dot u+u-f(u,\dot u),
\qquad
\cL_1[\varphi]=\ddot\varphi+2\dot\varphi+\varphi.
\]
The solution of the problem $\cL_1[\varphi]=g$ with initial conditions \eqref{eq:25} has the form
\[
\varphi(t)=(u_0+v_0t)e^{-t}+\int_0^t(t-\tau)e^{-(t-\tau)}g(\tau)\,d\tau.
\]
For zero initial data and a bounded continuous function $g$, the estimates
\begin{equation}
\label{eq:26}
\|\varphi\|_\infty\le\|g\|_\infty,
\qquad
\|\dot\varphi\|_\infty\le2\|g\|_\infty
\end{equation}
hold.

For $\hbar\ne0$ introduce the deformation equation
\[
(1-p)\cL_1[U]=\hbar p\cN_1[U],
\]
\[
U(0,p,\hbar)=u_0,\qquad \dot U(0,p,\hbar)=v_0-u_0.
\]
We seek the solution as a formal series
\begin{equation}
\label{eq:27}
U(t,p,\hbar)=\sum_{k=0}^{\infty}u^{[k]}(t,\hbar)p^k,
\qquad
u^{[k]}=\left.\frac{1}{k!}\frac{\partial^kU}{\partial p^k}\right|_{p=0}.
\end{equation}
Here
\[
u^{[0]}(t)=(u_0+v_0t)e^{-t},
\qquad
u^{[k]}(0,\hbar)=\dot u^{[k]}(0,\hbar)=0,\quad k\ge1.
\]
Let
\[
f(U,\dot U)=\sum_{k=0}^{\infty}F^{[k]}(t,\hbar)p^k.
\]
Comparing coefficients of like powers of $p$ gives
\[
u^{[k]}(t,\hbar)=\int_0^t(t-\tau)e^{-(t-\tau)}R_k(\tau)\,d\tau,
\qquad k\ge1,
\]
where
\[
R_1=-\hbar F^{[0]},\qquad
R_k=(1+\hbar)\cL_1[u^{[k-1]}]-\hbar F^{[k-1]},\quad k\ge2.
\]
For $t\ge0$ we have
\begin{equation}
\label{eq:28}
|u^{[0]}(t)|\le |u_0|+\frac{|v_0|}{e}=:U_0,
\qquad
|\dot u^{[0]}(t)|\le |v_0-u_0|+\frac{|v_0|}{e}=:V_0.
\end{equation}
It is further assumed that $U_0<r_1$, $V_0<r_2$. For fixed $\hbar$ set
\[
M_k=\|u^{[k]}\|_\infty,
\qquad
N_k=\|\dot u^{[k]}\|_\infty,
\qquad
\rho_k=\|R_k\|_\infty,
\qquad k\ge1.
\]

From \eqref{eq:26} it follows that
\begin{equation}
\label{eq:29}
M_k\le\rho_k,\qquad N_k\le2\rho_k,\qquad k\ge1.
\end{equation}
Set
\[
U=u^{[0]}+\xi_1,\qquad \dot U=\dot u^{[0]}+\eta_1,
\]
where
\[
\xi_1=\sum_{k=1}^{\infty}u^{[k]}(t,\hbar)p^k,
\qquad
\eta_1=\sum_{k=1}^{\infty}\dot u^{[k]}(t,\hbar)p^k.
\]
Define
\[
\widetilde f(U,V)=\sum_{\mathfrak n,\mathfrak m\ge0}
|a_{\mathfrak n,\mathfrak m}|U^{\mathfrak n}V^{\mathfrak m},
\qquad
\Psi(z)=\sum_{k=1}^{\infty}\rho_k z^k.
\]
Then
\[
\widehat F(z)=\sum_{k=0}^{\infty}\widehat F^{[k]}z^k
\preceq
\widetilde f(U_0+\Psi(z),V_0+2\Psi(z)),
\]
where
\[
\widehat F^{[k]}=\sup_{t\ge0}|F^{[k]}(t,\hbar)|.
\]
Let $-2<\hbar<0$,
\[
\kappa=\frac{|\hbar|}{1-\upsilon},
\]
and
\[
\widetilde T(V)=\kappa\widetilde f(U_0+V,V_0+2V),
\qquad
V_{\max}=\min\left\{r_1-U_0,\frac{r_2-V_0}{2}\right\}.
\]

\begin{lemma}
If there exists $V^*\in(0,V_{\max})$, for which
\[
\widetilde T(V^*)\le V^*,
\]
then
\[
\sum_{k=1}^{\infty}\rho_k\le V^*.
\]
\end{lemma}

\begin{proof}
Set
\[
S_m=\sum_{k=1}^{m}\rho_k,\qquad S_0=0.
\]
From the recurrence estimates we have
\[
S_m\le\upsilon S_{m-1}+|\hbar|\sum_{k=0}^{m-1}\widehat F^{[k]}.
\]
The coefficientwise majorant for $\widehat F$ and the non-negativity of the coefficients of $\widetilde f$ give, for $S_{m-1}\le V^*$,
\[
\sum_{k=0}^{m-1}\widehat F^{[k]}
\le\widetilde f(U_0+V^*,V_0+2V^*).
\]
Hence $S_m\le V^*$. By induction, $S_m\le V^*$ for all $m\ge1$. Passing to the limit as $m\to\infty$, gives the required inequality.
\end{proof}

For $A\in(0,V_{\max})$ denote
\[
\zeta_A=\kappa\left[
\widetilde f_U(U_0+A,V_0+2A)+
2\widetilde f_V(U_0+A,V_0+2A)
\right].
\]

\begin{theorem}
Let $-2<\hbar<0$ and suppose there exists $A\in(0,V_{\max})$, for which
\begin{equation}
\label{eq:30}
\zeta_A<1,
\end{equation}
\begin{equation}
\label{eq:31}
\kappa\widetilde f(U_0,V_0)\le(1-\zeta_A)A.
\end{equation}
Then, at $p=1$ series \eqref{eq:27}, together with the series of its first and second derivatives, converges absolutely and uniformly on $[0,+\infty)$, and the sum of the series is the unique solution of problem \eqref{eq:24}, \eqref{eq:25}.
\end{theorem}

\begin{proof}
From conditions \eqref{eq:30}, \eqref{eq:31} we obtain $\widetilde T(A)\le A$. By Lemma~11,
\[
\sum_{k=1}^{\infty}\rho_k<\infty.
\]
From \eqref{eq:29} and
\[
\ddot u^{[k]}=R_k-2\dot u^{[k]}-u^{[k]},\qquad k\ge1,
\]
we get
\[
\|\ddot u^{[k]}\|_\infty\le6\rho_k,
\]
\[
\sum_{k=0}^{\infty}\|F^{[k]}\|_\infty
\le\widetilde f(U_0+A,V_0+2A)<\infty,
\qquad
\sum_{k=0}^{\infty}F^{[k]}=f(u,\dot u).
\]
Passing to the limit in the recurrence equations gives
\[
\cL_1[u]=(1+\hbar)\cL_1[u]-\hbar f(u,\dot u),
\]
whence $\cL_1[u]=f(u,\dot u)$ for $\hbar\ne0$. The initial conditions hold by construction, and uniqueness follows from the analyticity of $f$.
\end{proof}

For $-1\le\hbar<0$ we have $\kappa=1$, whereas for $-2<\hbar<-1$ we have $\kappa>1$. Hence, within this majorant estimate, the interval $[-1,0)$ gives the least restrictive sufficient conditions for convergence.

For problem \eqref{eq:22},
\[
f(u,\dot u)=-u\dot u-u^2,
\qquad
\widetilde f(U,V)=UV+U^2
\]
we have
\[
F^{[k]}=-\sum_{j=0}^{k}u^{[j]}
\left(\dot u^{[k-j]}+u^{[k-j]}\right).
\]
The first terms of the series have the form
\[
u^{[0]}(t)=(u_0+v_0t)e^{-t},
\]
\[
u^{[1]}(t,\hbar)=\hbar v_0\left[P(t)e^{-t}+Q(t)e^{-2t}\right],
\]
where
\[
P(t)=-u_0-2v_0+(u_0+v_0)t,
\qquad
Q(t)=u_0+2v_0+v_0t.
\]
Set
\[
c_2=-\frac{2u_0^2+6u_0v_0+3v_0^2}{4},
\qquad
c_3=\frac{6u_0^2+28u_0v_0+17v_0^2}{8},
\]
\[
D_1(t)=u_0^2+4u_0v_0+2v_0^2+2v_0(u_0+v_0)t,
\]
\[
K_1(t)=\frac{2u_0^2+4u_0v_0-v_0^2}{8}
+\frac{u_0v_0+v_0^2}{2}t+\frac{v_0^2}{4}t^2,
\]
\[
A_2(t)=(1+\hbar)\hbar v_0P(t)-\hbar^2v_0(c_3+c_2t),
\]
\[
A_3(t)=(1+\hbar)\hbar v_0Q(t)+\hbar^2v_0D_1(t).
\]
Then
\[
u^{[2]}(t,\hbar)=A_2(t)e^{-t}+A_3(t)e^{-2t}
-\hbar^2v_0K_1(t)e^{-3t}.
\]
For $-1\le\hbar<0$ conditions \eqref{eq:30}, \eqref{eq:31} take the form
\[
4U_0+V_0+6A<1,
\]
\[
U_0V_0+U_0^2\le[1-4U_0-V_0-6A]A.
\]
Setting
\[
a_2=1-4U_0-V_0,
\qquad
a_3=24(U_0V_0+U_0^2),
\]
for $a_2>0$, $a_2^2\ge a_3>0$ we obtain
\[
A\in\left[
\frac{a_2-\sqrt{a_2^2-a_3}}{12},
\frac{a_2+\sqrt{a_2^2-a_3}}{12}
\right]\cap(0,V_{\max}).
\]
In the case under consideration $V_{\max}=+\infty$. For $U_0=V_0=0$ we have the trivial solution $u(t)\equiv0$.

Thus the Cauchy problem for system \eqref{eq:20}, under the stated conditions, admits a convergent homotopy representation of its solution on $[0,+\infty)$.
\end{example}

\begin{example}
Consider the Cauchy problem for system \eqref{eq:21} with
\[
u_0=-0.05,\qquad v_0=-0.03.
\]
From \eqref{eq:28} we have
\[
U_0\approx0.0610363832,\qquad V_0\approx0.0310363832.
\]
From conditions \eqref{eq:30}, \eqref{eq:31} we obtain
\[
\hbar_{\min}\le\hbar<0,
\qquad
\hbar_{\min}\approx-1.21770410.
\]
For each fixed $\hbar$ from this interval, the homotopy series of problem \eqref{eq:22}, \eqref{eq:23}, together with the series of its first and second derivatives, converges absolutely and uniformly on $[0,+\infty)$.

For a practical choice of $\hbar$ consider the second-order approximation
\[
u_2(t,\hbar)=u^{[0]}(t)+u^{[1]}(t,\hbar)+u^{[2]}(t,\hbar)
\]
and set, at $t=1$,
\[
\Psi_1(\hbar)=u_2(1,\hbar),
\]
and additionally consider
\[
\Psi_2(\hbar)=\left.\frac{\partial u_2(t,\hbar)}{\partial t}\right|_{t=1}.
\]
The functions $\Psi_1(\hbar)$, $\Psi_2(\hbar)$ attain their minima, respectively, at
\[
\hbar_{\Psi_1}^*\approx-1.015527,
\qquad
\hbar_{\Psi_2}^*\approx-1.032126.
\]
For $-1.231<\hbar<-0.739$ their values, rounded to four decimal places, equal $-0.0297$ and $0.0182$ respectively. Taking into account the conditions of Theorem~3, we choose the working interval
\[
\hbar_{\min}\le\hbar<-0.739.
\]

\begin{figure}[htbp]
\centering
\includegraphics[width=0.58\textwidth]{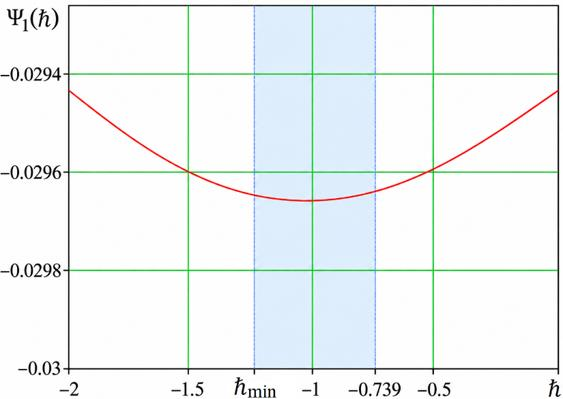}
\caption{The $\hbar$-curve $\Psi_1(\hbar)=u_2(1,\hbar)$, obtained at the second order of approximation. The working interval $[\hbar_{\min},-0.739)$ is highlighted.}
\label{fig:hcurve}
\end{figure}

For $\hbar=-1$ the second-order approximation has the form
\[
\begin{aligned}
u_2(t,-1)\approx{}&-(0.046428875+0.03252525t)e^{-t}\\
&-(0.003609+0.001044t)e^{-2t}\\
&+(0.000037875+0.000036t+0.00000675t^2)e^{-3t}.
\end{aligned}
\]
The original variables are recovered from the formulas
\[
x_1(t)\approx1+u_2(t,-1),
\qquad
y_1(t)\approx1+u_2(t,-1)+\dot u_2(t,-1).
\]
Figure~\ref{fig:numcomp} shows their comparison with the numerical solutions $X_1,Y_1$ of system \eqref{eq:20}, obtained using MathCAD with initial conditions $x_1(0)=0.95$, $y_1(0)=0.97$.

\begin{figure}[htbp]
\centering
\includegraphics[width=0.94\textwidth]{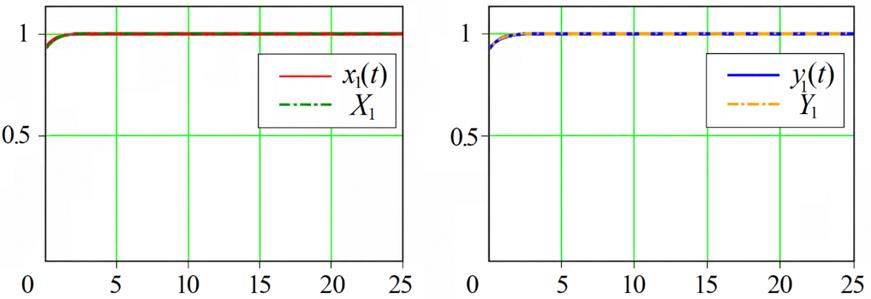}
\caption{Graphs of the numerical and approximate analytical solutions of system \eqref{eq:20}.}
\label{fig:numcomp}
\end{figure}

In the numerical comparison on the interval $[0,7]$ the absolute discrepancy for both components is less than $2.2402\cdot10^{-7}$. On the interval $(7,25]$ it does not exceed $2.3290\cdot10^{-8}$.
\end{example}

Let us now consider an example in which the homotopy series admits exact summation.

\begin{example}
We study the Cauchy problem for the continuous analogue of $W_1$ \cite{ref10}. On $\Om_2$ the system has the form
\[
\begin{cases}
\dot x_1=x_2y_1,\\
\dot x_2=x_2y_2-x_2,\\
\dot y_1=x_1+x_2y_1-y_1,\\
\dot y_2=x_2y_2-y_2,
\end{cases}
\quad
(x_1(0),x_2(0),y_1(0),y_2(0))=(x_1^0,x_2^0,y_1^0,y_2^0)\in\Om_2.
\]
For this problem an exact analytical solution was obtained in \cite{ref26}.

Taking into account the simplex conditions, we get
\[
x_2(t)-y_2(t)=c_4e^{-t},
\qquad c_4=x_2^0-y_2^0.
\]
Consequently, $y_2$ satisfies the Riccati equation
\[
\dot y_2+(1-c_4e^{-t})y_2-y_2^2=0,
\qquad y_2(0)=y_2^0.
\]
Set
\[
\cL_2[\varphi]=\dot\varphi+(1-c_4e^{-t})\varphi,
\qquad
E_2(t)=\exp[-t+c_4(1-e^{-t})].
\]
At $\hbar=-1$ the deformation equation takes the form
\[
\cL_2[Y_2]=pY_2^2,
\qquad Y_2(0,p)=y_2^0.
\]
\[
Y_2(t,p)=\sum_{k=0}^{\infty}y_2^{[k]}(t)p^k,
\qquad y_2^{[0]}(t)=y_2^0E_2(t),
\]
and the recurrence equations have the form
\[
\cL_2[y_2^{[k]}]
=\sum_{j=0}^{k-1}y_2^{[j]}y_2^{[k-1-j]},
\qquad y_2^{[k]}(0)=0,\quad k\ge1.
\]
Denote
\[
I_3(t)=\int_0^tE_2(s)\,ds
=\begin{cases}
\displaystyle\frac{\exp[c_4(1-e^{-t})]-1}{c_4},&c_4\ne0,\\[2ex]
1-e^{-t},&c_4=0.
\end{cases}
\]
Since
\[
\cL_2[E_2(t)I_3^k(t)]=kE_2^2(t)I_3^{k-1}(t),\qquad k\ge1,
\]
by induction
\[
y_2^{[k]}(t)=(y_2^0)^{k+1}E_2(t)I_3^k(t),
\qquad k=0,1,\ldots.
\]
For interior initial data,
\[
E_2(t)\le e^{-[1-\max\{c_4,0\}]t},
\qquad
\sup_{t\ge0}y_2^0I_3(t)
\le\frac{y_2^0}{1-\max\{c_4,0\}}<1.
\]
Hence, for $p=1$ the geometric series converges absolutely and uniformly on $[0,+\infty)$, and
\[
y_2(t)=\sum_{k=0}^{\infty}y_2^{[k]}(t)
=\frac{y_2^0E_2(t)}{1-y_2^0I_3(t)}.
\]
The remaining components are recovered from the equalities
\[
x_2(t)=y_2(t)+c_4e^{-t},
\qquad x_1(t)=1-x_2(t),
\qquad y_1(t)=1-y_2(t).
\]
This formula coincides with the corresponding exact solution from \cite{ref26}.

An analogous construction at $\hbar=-1$ applies to the problem
\[
\dot u(t)+f_1(t)u(t)-\beta_1u^2(t)=0,
\qquad u(0)=u_0,
\]
where
\[
f_1(t)\in C[0,T],\qquad \beta_1,u_0\in\R.
\]
Set
\[
E_3(t)=\exp\left(-\int_0^t f_1(s)\,ds\right),
\qquad
I_4(t)=\int_0^tE_3(s)\,ds.
\]
If
\[
|\beta_1u_0|I_4(T)<1,
\]
then the homotopy series converges uniformly on $[0,T]$, and its sum equals
\[
u(t)=\frac{u_0E_3(t)}{1-\beta_1u_0I_4(t)}.
\]
\end{example}

\section{Explicit analytical solutions of system \eqref{eq:5} in special cases}

Under special relations among the heredity coefficients, the Cauchy problem for system \eqref{eq:5} admits explicit analytical solutions. Let $z^0\in\Om$, $t\ge0$. Auxiliary notation is introduced separately in each case.

For $1\le j,k\le n$, $j\ne k$, set
\[
a_{jk}^{(f)}=p_{jk,j}^{(f)}+p_{kj,j}^{(f)}-1,
\qquad
a_{jk}^{(m)}=p_{jk,j}^{(m)}+p_{kj,j}^{(m)}-1,
\]
and set $a_{jj}^{(f)}=a_{jj}^{(m)}=0$.

\medskip
\noindent\textbf{Case 1 (single-block class).}
Let
\[
n=\nu=N_0,\qquad x_j^0=y_j^0=s_j^0,\qquad s^0\in S^{N_0-1},
\]
and for $j\ne k$,
\[
a_{jk}^{(f)}=a_{jk}^{(m)}=\mathfrak q_j- \mathfrak q_k,
\qquad |\mathfrak q_j- \mathfrak q_k|\le1.
\]
Then
\[
x_j(t)=y_j(t)=
\frac{s_j^0e^{\mathfrak q_jt}}{\displaystyle\sum_{k=1}^{N_0}s_k^0e^{\mathfrak q_kt}},
\qquad j=1,\ldots,N_0.
\]

\medskip
\noindent\textbf{Case 2 (two interacting blocks).}
Let
\[
n=\nu=N_0,\qquad
\{1,\ldots,N_0\}=\Lambda\sqcup\Pi,
\qquad \Lambda,\Pi\ne\varnothing,
\]
\[
x_j^0=y_j^0=s_j^0,\qquad s^0\in S^{N_0-1},
\]
and for $i\ne k$,
\[
a_{ik}^{(f)}=a_{ik}^{(m)}=a_{ik},
\]
where
\begin{equation}
\label{eq:32}
a_{ik}=
\begin{cases}
\mathfrak q_i- \mathfrak q_k,&i,k\in\Lambda,\\
\mathfrak r_i- \mathfrak r_k,&i,k\in\Pi,\\
\sigma,&i\in\Lambda,\ k\in\Pi,\\
-\sigma,&i\in\Pi,\ k\in\Lambda,
\end{cases}
\end{equation}
with
\begin{equation}
\label{eq:33}
|\mathfrak q_i- \mathfrak q_k|\le1\quad(i,k\in\Lambda),\qquad
|\mathfrak r_i- \mathfrak r_k|\le1\quad(i,k\in\Pi),\qquad
|\sigma|\le1.
\end{equation}
Denote
\[
U_{01}=\sum_{i\in\Lambda}s_i^0
\]
and suppose $0<U_{01}<1$, $\sigma\ne0$. Then
\[
\begin{cases}
\displaystyle
x_i(t)=y_i(t)=U(t)
\frac{s_i^0e^{\mathfrak q_i\tau_\Lambda(t)}}{\displaystyle\sum_{\ell\in\Lambda}s_\ell^0e^{\mathfrak q_\ell\tau_\Lambda(t)}},
&i\in\Lambda,\\[3ex]
\displaystyle
x_j(t)=y_j(t)=[1-U(t)]
\frac{s_j^0e^{\mathfrak r_j\tau_\Pi(t)}}{\displaystyle\sum_{\ell\in\Pi}s_\ell^0e^{\mathfrak r_\ell\tau_\Pi(t)}},
&j\in\Pi,
\end{cases}
\]
where
\[
U(t)=\frac{U_{01}e^{\sigma t}}{1-U_{01}+U_{01}e^{\sigma t}},
\]
\[
\tau_\Lambda(t)=\frac1\sigma\ln\left(1-U_{01}+U_{01}e^{\sigma t}\right),
\qquad
\tau_\Pi(t)=t-\tau_\Lambda(t).
\]
For $\sigma=0$,
\[
U(t)=U_{01},\qquad
\tau_\Lambda(t)=U_{01}t,
\qquad
\tau_\Pi(t)=(1-U_{01})t.
\]

\medskip
\noindent\textbf{Case 3 (macro-group subcase with a double root).}
Let
\[
n=\nu=N_0,\qquad
\{1,\ldots,N_0\}=\Lambda\sqcup\Pi,
\qquad \Lambda,\Pi\ne\varnothing,
\qquad a,b,c,d\in[0,1].
\]
For $i,k$ within the same group set
\[
p_{ik,i}^{(f)}=p_{ik,k}^{(m)}=1.
\]
And for $i\in\Lambda$, $k\in\Pi$,
\[
p_{ik,i}^{(f)}=a,\qquad p_{ik,i}^{(m)}=c,
\]
and for $i\in\Pi$, $k\in\Lambda$,
\[
p_{ik,k}^{(f)}=b,\qquad p_{ik,k}^{(m)}=d.
\]
Suppose that for some $\varkappa>0$,
\[
1-a=\varkappa c,\qquad
b=\varkappa(1-d),\qquad
1-a\ne b.
\]
Set
\[
U_{02}=\sum_{i\in\Lambda}x_i^0,
\qquad
V_{02}=\sum_{i\in\Lambda}y_i^0,
\]
and
\[
x_i^0=\alpha_{0i}U_{02},\qquad
y_i^0=\alpha_{0i}V_{02}\quad(i\in\Lambda),
\]
\[
x_j^0=\beta_{0j}(1-U_{02}),\qquad
y_j^0=\beta_{0j}(1-V_{02})\quad(j\in\Pi),
\]
where
\[
\alpha_{0i},\beta_{0j}\ge0,
\qquad
\sum_{i\in\Lambda}\alpha_{0i}=\sum_{j\in\Pi}\beta_{0j}=1.
\]
Denote
\[
H_4=(1-a)(\varkappa-U_{02}-\varkappa V_{02})
+b(U_{02}+\varkappa V_{02}+1),
\qquad
r_3=\frac{H_4}{2(a+b-1)}.
\]
If
\[
H_4^2-4b(a+b-1)(U_{02}+\varkappa V_{02})=0,
\]
then

where
\[
U(t)=r_3+
\frac{U_{02}-r_3}{1-\dfrac{b+a-1}{\varkappa}(U_{02}-r_3)t},
\qquad
V(t)=\frac{U_{02}+\varkappa V_{02}-U(t)}{\varkappa}.
\]

\medskip
\noindent\textbf{Case 4 (constant background of additional male types).}
Let
\[
n<\nu,\qquad
E_3=\sum_{l=n+1}^{\nu}y_l^0,
\qquad
0<E_3<1,
\qquad
y_j^0=(1-E_3)x_j^0,\quad j=1,\ldots,n.
\]
For $l>n$ set
\[
p_{il,i}^{(f)}=1,\qquad p_{il,l}^{(m)}=1,
\qquad i=1,\ldots,n.
\]
For $1\le j,k\le n$, $j\ne k$, suppose
\[
a_{jk}^{(f)}=\mathfrak q_j- \mathfrak q_k,
\qquad
a_{jk}^{(m)}=(1-E_3)(\mathfrak q_j-\mathfrak q_k),
\qquad |q_j-q_k|\le1.
\]
Then
\[
\begin{cases}
\displaystyle
x_j(t)=\frac{x_j^0e^{(1-E_3)\mathfrak q_jt}}
{\displaystyle\sum_{\ell=1}^{n}x_\ell^0e^{(1-E_3)\mathfrak q_\ell t}},
&j=1,\ldots,n,\\[3ex]
y_j(t)=(1-E_3)x_j(t),&j=1,\ldots,n,\\
y_l(t)=y_l^0,&l=n+1,\ldots,\nu.
\end{cases}
\]

\medskip
\noindent\textbf{Case 5 (basic types, a distinguished type, and additional male types).}
Let $n<\nu$, $n\ge3$,
\[
\{1,\ldots,n-1\}=\Lambda\sqcup\Pi,
\qquad \Lambda,\Pi\ne\varnothing,
\qquad x_i^0=y_i^0=s_i^0,\quad i<n,
\]
\[
\sum_{i=1}^{n-1}s_i^0=U_{03},
\qquad 0<U_{03}<1,
\qquad x_n^0=1-U_{03},
\qquad y_n^0+\sum_{l=n+1}^{\nu}y_l^0=1-U_{03}.
\]
For $1\le i,k<n$ define the matrix $a_{ik}$ according to \eqref{eq:32}, \eqref{eq:33}. For $i<k<n$ set
\[
p_{ik,i}^{(f)}=p_{ki,i}^{(f)}=p_{ik,i}^{(m)}=p_{ki,i}^{(m)}
=\frac{1+a_{ik}}{2},
\]
\[
p_{ik,k}^{(f)}=p_{ki,k}^{(f)}=p_{ik,k}^{(m)}=p_{ki,k}^{(m)}
=\frac{1-a_{ik}}{2}.
\]
Interactions with type $n$ are defined by
\[
p_{in,i}^{(f)}=p_{in,n}^{(m)}=p_{ni,n}^{(f)}=p_{ni,i}^{(m)}=1,
\qquad i<n.
\]
For $l>n$ set
\[
p_{il,i}^{(f)}=p_{il,l}^{(m)}=1,
\qquad i<n,
\]
\[
p_{nl,n}^{(f)}=1,
\qquad
p_{nl,l}^{(m)}=1-\eta_l,
\qquad
p_{nl,n}^{(m)}=\eta_l,
\qquad
0<\eta_l\le1.
\]
The remaining probabilities are determined by normalization and the Volterra condition.

Let
\[
U_{01}=\sum_{i\in\Lambda}s_i^0,
\qquad 0<U_{01}<U_{03}\quad\text{and}\quad\sigma\ne0.
\]
Then
\[
\begin{cases}
\displaystyle
x_i(t)=y_i(t)=\overline U(t)
\frac{s_i^0e^{\mathfrak q_i\tau_\Lambda(t)}}
{\displaystyle\sum_{\ell\in\Lambda}s_\ell^0e^{\mathfrak q_\ell\tau_\Lambda(t)}},
&i\in\Lambda,\\[3ex]
\displaystyle
x_i(t)=y_i(t)=[U_{03}-\overline U(t)]
\frac{s_i^0e^{\mathfrak r_i\tau_\Pi(t)}}
{\displaystyle\sum_{\ell\in\Pi}s_\ell^0e^{\mathfrak r_\ell\tau_\Pi(t)}},
&i\in\Pi,\\[3ex]
x_n(t)=1-U_{03},&\\
y_l(t)=y_l^0e^{-\eta_l(1-U_{03})t},&l>n,\\[1ex]
\displaystyle
y_n(t)=y_n^0+\sum_{l=n+1}^{\nu}y_l^0
\left(1-e^{-\eta_l(1-U_{03})t}\right).&
\end{cases}
\]
where
\[
\overline U(t)=
\frac{U_{03}U_{01}e^{\sigma U_{03}t}}
{U_{03}-U_{01}+U_{01}e^{\sigma U_{03}t}},
\]
\[
\tau_\Lambda(t)=\frac1\sigma
\ln\left(
\frac{U_{03}-U_{01}+U_{01}e^{\sigma U_{03}t}}{U_{03}}
\right),
\qquad
\tau_\Pi(t)=U_{03}t-\tau_\Lambda(t).
\]
For $\sigma=0$:
\[
\overline U(t)=U_{01},
\qquad
\tau_\Lambda(t)=U_{01}t,
\qquad
\tau_\Pi(t)=(U_{03}-U_{01})t.
\]
Note that the cases presented do not exhaust all explicitly integrable special cases of system \eqref{eq:5}.

\section{Conclusion}

The sufficient criterion for absolute and uniform convergence of the homotopy series on $[0,+\infty)$, obtained under the Hurwitz condition for the reduced Jacobian matrix, requires neither reality of the eigenvalues nor their linear independence over $\mathbb{Q}$. For the continuous analogue of $W_1$ \cite{ref10} at $\hbar=-1$ and interior initial data, the series admits exact summation and reproduces the known analytical solution.

The majorant scheme is applicable to the quadratic-linear systems considered here, as well as to problems with analytic nonlinearity, provided the corresponding convergence conditions are satisfied. These conditions may restrict the initial data and parameters; however, their failure does not by itself imply divergence of the series. The explicit solutions obtained for special cases are of independent significance and do not rely on the majorant conditions.

The continuous analogue should be regarded as an independent model, rather than as an unconditionally equivalent substitute for the discrete VQSOBP. Consequently, the transfer of conclusions concerning stability, limiting behavior, and periodic regimes between the two models requires separate justification. The conditions established on the coefficients allow one to compare discrete and continuous dynamics in the cases considered, but they do not provide a complete classification for arbitrary coefficients.

\end{document}